\documentclass[12pt, doublespace]{amsart}

\usepackage{amssymb,amsfonts}
\usepackage{mathtools}
\usepackage{amsmath,amsthm}
\usepackage{mathrsfs}

\usepackage{comment}

\usepackage[all]{xy}
\usepackage[active]{srcltx}
\usepackage{epic,eepic,eepicemu}
\usepackage[usenames,dvipsnames]{xcolor}
\usepackage[normalem]{ulem}
\usepackage{enumerate}

\usepackage[margin=3.5cm]{geometry}

\allowdisplaybreaks

\numberwithin{equation}{section}

\definecolor{dgreen}{rgb}{0.1,0.7,0.1}
\definecolor{purple}{rgb}{0.49, 0.06, 0.51}

\newcommand{\Q}{\mathbb{Q}}
\def\N{\mathbb{N}}
\def\R{\mathbb{R}}
\def\Z{\mathbb{Z}}

\newcommand{\CC}{\mathscr{C}}
\newcommand{\CP}{\mathscr{P}}
\newcommand{\CQ}{\mathscr{Q}}

\newcommand{\CS}{\mathscr{S}}

\def\o{\overline}

\def\s{\sigma}
\def\m#1{\begin{pmatrix}#1\end{pmatrix}}

\newcommand{\vt}{\vartheta}
\newcommand{\ox}{\otimes}
\newcommand{\x}{\times}
\newcommand{\sm}{\setminus}

\newcommand{\Nil}{\mathrm{Nil}}
\newcommand{\wt}{\widetilde}
\newcommand{\id}{\mathrm{id}}
\newcommand{\vf}{\varphi}

\newcommand{\gr}{\mathrm{gr}}

\newcommand{\ovl}{\overline}

\newcommand{\pri}{\prescript{\iota}{}}

\newcommand{\bt}{\overline{\raisebox{-.3ex}{\phantom{x}}}^{\,t}}
\newcommand{\bbar}{\overline{\raisebox{-.3ex}{\phantom{x}}}\,}
\newcommand{\bbarr}{\overline{\raisebox{-.3ex}{\phantom{x}}}}

\DeclareMathOperator{\lift}{Lift}

\DeclareMathOperator{\Sym}{Sym}

\DeclareMathOperator{\st}{st}

\DeclareMathOperator{\Int}{Int}

\DeclareMathOperator{\Trd}{Trd}

\DeclareMathOperator{\diag}{diag}

\DeclareMathOperator{\End}{End}

\newcommand{\PSD}{\mathrm{PSD}}

\renewcommand{\geq}{\geqslant}
\renewcommand{\leq}{\leqslant}

\renewcommand{\ge}{\geqslant}
\renewcommand{\le}{\leqslant}

\newcommand{\ve}{\varepsilon}

\newcommand{\op}{\mathrm{op}}

\newcommand{\pr}{\mathrm{prod}}

\newcommand{\ad}{\mathrm{ad}}

\newcommand{\tas}{T_{(A,\s)}}

\newcommand{\qf}[1]{\langle #1\rangle}

\newcommand{\simtoo}{\overset\sim\longrightarrow}

\DeclareMathOperator{\ch}{\mathrm{char}}

\newcommand{\tu}{\textup}

\newcommand{\lra}{\Leftrightarrow}

\newcommand{\kk}{k}

\newtheorem{thm}{Theorem}[section]
\newtheorem{prop}[thm]{Proposition}
\newtheorem{cor}[thm]{Corollary}
\newtheorem{lemma}[thm]{Lemma}
\theoremstyle{definition}
\newtheorem{defi}[thm]{Definition}

\newtheorem{remark}[thm]{Remark}
\newtheorem{ex}[thm]{Example}

\begin{document}

\title{Positive cones and gauges on algebras with involution}
\author{Vincent Astier}
\address{School of Mathematics and Statistics\\ University College Dublin\\ Belfield\\
Dublin~4\\ Ireland} 
\email{vincent.astier@ucd.ie, thomas.unger@ucd.ie}

\author{Thomas Unger}



\maketitle

\begin{abstract}
We extend the classical links between valuations and orderings on fields to
Tignol-Wadsworth gauges and positive cones on finite-dimensional simple
algebras with involution. We also study the compatibility of gauges and
positive cones, and prove lifting results in the style of the Baer-Krull
theorem for fields.
\end{abstract}

\section{Introduction}

The links between valuations and orderings on fields are well-known
and are an essential tool in real algebra and in quadratic form
theory, see for instance \cite{BCR, Lam-1983, Prestel84}.

In the recent papers \cite{T-W-2010, T-W-2011} and the book
\cite{T-W-book}, Tignol and Wadsworth introduced and started the
development of the theory of ``gauges'', valuation-like functions on
(finite-dimensional) semisimple algebras and central simple algebras
with involution, while in the past few years we studied signatures of
hermitian forms and their links to positivity \cite{A-U-Kneb,
A-U-prime, A-U-PS} and used them to propose ``positive cones'', a
notion of ordering for (finite-dimensional) simple algebras with
involution, cf. \cite{A-U-pos}.

It is therefore natural to wonder if there could be links between
gauges and positive cones, and if so, if they would be similar to the
classical ones in the field case between valuations and orderings.
This paper presents a positive answer to both questions. We show
that, on a finite-dimensional simple algebra $A$ with involution
$\s$, the most natural construction of a ``valuation ring''
associated to a positive cone leads to a unique $\s$-special gauge
(Section~\ref{sec:gauges-pc}). We then show that a notion of
compatibility between $\s$-invariant gauges and positive cones can be
described by several equivalent conditions, reminiscent of the field
case (Section~\ref{sec:comp}). We proceed with a number of technical
results about anisotropy of forms for use in the final section
(Section~\ref{sec:aniso}) and conclude the paper by studying the
lifting of positive cones from the residue algebra of a
$\s$-invariant gauge, in the style of the Baer-Krull theorem for
fields (Section~\ref{sec:B-K}).

\section{Notation}

In this paper $F$ will always be a field of characteristic different
from $2$, $A$ will always be an $F$-algebra, and $\s$ an $F$-linear
involution. We will explicitly indicate when $A$ is
finite-dimensional over $F$, simple or semisimple, and also when
$F=\Sym(A,\s)\cap Z(A)$, where $\Sym(A,\s):=\{a\in A \mid \s(a)=a\}$
and $Z(A)$ denotes the centre of $A$.

For a subset $S\subseteq A$ we define $S^\x:= S\cap A^\x$, where
$A^\x$ is the set of invertible elements of $A$. For each $n\in \N$
we denote the involution $(a_{ij})_{i,j}\mapsto (\s(a_{ji}))_{i,j}$
on $M_n(A)$ by $\s^t$. For $a_1,\ldots, a_\ell \in \Sym(A,\s)$, we
denote the hermitian form $A^\ell \x A^\ell \to A, (x, y)\mapsto
\s(x_1)a_1y_1+\ldots +\s(x_\ell) a_\ell y_\ell$ by $\qf{a_1,\ldots,
a_\ell}_\s$. For $a\in A^\x$, we denote the inner automorphism $A\to
A, x\mapsto axa^{-1}$ by $\Int(a)$. If $h$ is a hermitian form over
$(A,\s)$, we denote by $D_{(A,\s)}(h)$ the set of elements of $A$
that are represented by $h$.

Let $X_F$ denote the space of orderings of $F$. We make the convention that
orderings in $X_F$ always contain $0$. If $P \in X_F$, we denote by $F_P$ a
real closure of $F$ at $P$.

For $P\in X_F$ and ${\kk}$ a subfield of $F$ we denote the convex closure of
${\kk}$ in $F$ with respect to $P$, $\{x\in F \mid \exists m\in {\kk}\ \
-m\leq_P x \leq_P m\}$, by $R_{{\kk},P}$. By $v_{{\kk},P}$ we denote the
valuation on $F$ with valuation ring $R_{v_{{\kk},P}}=R_{{\kk},P}$. We also
denote the unique maximal ideal of $R_{{\kk},P}$ by $I_{{\kk},P}$. Recall that
$I_{\kk,P}=\{x\in F \mid \forall \ve\in \kk^\x \cap P \ -\ve\leq_P x \leq_P
\ve\}$.

If $v$ is a valuation on $F$, we denote its value group $v(F^\x)$ by $\Gamma_v$
and its residue field by $F_v$. In the context of graded rings later in the
paper, we will also use the notation $F_0$ for $F_v$. Recall from
\cite[p.~72]{Prestel84} that $v$ and $P$ are called compatible if for all
$a,b\in F$, $0<_Pa\leq_P b$ implies $v(a) \geq v(b)$. See
\cite[Lemma~7.2]{Prestel84} for equivalent characterizations; one that we will
repeatedly use is $v(a+b)=\min\{ v(a), v(b)\}$ for all $a,b \in P$. Also recall
from \cite[Thm.~7.21]{Prestel84} that if $v$ is compatible with $P$, then
$v=v_{{\kk},P}$ for some subfield ${\kk}$ of $F$.

Several of our results use matrices with quaternion coefficients, cf. Zhang's
paper \cite{Zhang-1997}, from which we recall what we need in
Appendix~\ref{sec:quat-mat}. Throughout the paper we make the convention that
eigenvalue means right eigenvalue.

\section{Gauges}

Gauges were defined as a general notion of ``valuation map'' on
finite-dimensional semisimple $F$-algebras by Tignol and Wadsworth, cf.
\cite{T-W-2010, T-W-2011, T-W-book} from which we recall the main definitions
below. We assume for the remainder of this section that $A$ is a
finite-dimensional semisimple $F$-algebra.

\begin{defi}
  Let $v:F\to \Gamma_v \cup \{\infty\}$ be a valuation on $F$ and let $\Gamma$
  be a totally ordered abelian group, containing $\Gamma_v$. A map $w: A\to
  \Gamma\cup\{\infty\}$ is
  \begin{enumerate}[$(1)$]
    \item  a \emph{$v$-value function} on $A$ if for all $x,y\in A$ and 
    $\lambda \in F$, we have
    \[w(x)=\infty \Leftrightarrow  x=0;\ w(x+y) \geq \min\{w(x), w(y)\};\ 
    w(\lambda x)= v(\lambda) 
    + w(x);\]

    \item \emph{surmultiplicative} if  $w(1)=0$ and
    $w(xy) \geq w(x) + w(y)$, for all $x,y\in A$;

    \item a \emph{$v$-norm} if  $A$ has a \emph{splitting basis}, i.e., an 
    $F$-basis $\{e_1,\ldots, 
    e_m\}$ such  that
      \[w (\sum_{i=1}^m \lambda_i e_i) =\min_{1\leq i \leq m} (v(\lambda_i) +w 
      (e_i)),\quad \forall 
      \lambda_1, \ldots, \lambda_m \in F.\]
  \end{enumerate}
\end{defi}

Let $w$ be a surmultiplicative $v$-value function on $A$ and let $\gamma \in
\Gamma$. We consider the abelian groups
\[A_{\geq \gamma}:= \{ x\in A \mid w(x) \geq \gamma\},\ A_{> \gamma}:= \{ x\in 
A \mid w(x) > 
\gamma\},
\text{ and } A_\gamma:= A_{\geq \gamma}/ A_{> \gamma}\]
and define
\[\gr_w(A):= \bigoplus_{\gamma \in \Gamma} A_\gamma,\]
which is a graded $\gr_v(F)$-algebra. Recall that a \emph{homogeneous $2$-sided
ideal} $I$ of $\gr_w(A)$ is a $2$-sided ideal of $\gr_w(A)$ such that $I =
\bigoplus_{\gamma \in \Gamma} I_\gamma$ where $I_\gamma := A_\gamma \cap I$ and
that $\gr_w(A)$ is called \emph{\tu{(}graded\tu{)} semisimple} if it does not
contain any nonzero homogeneous two-sided nilpotent ideal.

If $w$ is a surmultiplicative $v$-value function on $A$, we write 
   \[R_{w} :=\{ a\in A \mid w(a) \geq 0\} \quad \text{ and  }\quad I_{w} :=\{ 
   a\in A \mid w(a) > 
   0\}\]
for $A_{\geq 0}$ and $A_{>0}$, respectively.

We also define $\Gamma_w:=w(A\sm\{0\})$ and note that it is a union of cosets
of $\Gamma_v$, and is not a group in general.

\begin{defi}
  Let $v:F\to \Gamma_v \cup \{\infty\}$ be a valuation on $F$ and let $\Gamma$
  be a totally ordered abelian group, containing $\Gamma_v$. A map $w: A\to
  \Gamma\cup\{\infty\}$ is a \emph{$v$-gauge }on $A$ if it is a
  surmultiplicative $v$-value function that is a $v$-norm and such that
  $\gr_w(A)$ is a semisimple $\gr_v(F)$-algebra.

  If $\s$ is an $F$-linear involution on $A$, then a gauge (and more generally
  a $v$-value function) $w$ on $A$ is called \emph{$\s$-invariant} if $w\circ
  \s =w$, and \emph{$\s$-special} if $w(\s(x)x)= 2 w(x)$ for all $x\in A$.
  (Note that \cite[Def.~1.2]{T-W-2011} has the additional assumption that $A$
  is simple.)
\end{defi}

\begin{remark}\label{spec-aniso}
  The existence of a $\s$-special surmultiplicative $v$-value function on $A$
  implies that $\s$ is anisotropic (i.e., $\s(x)x=0$ implies $x=0$ for every
  $x\in A$), as observed in the first line of the proof of
  \cite[Prop.~1.1]{T-W-2011}. Also, a $\s$-special surmultiplicative $v$-value
  function is $\s$-invariant, cf. \cite[Prop.~1.1]{T-W-2011}.
\end{remark}

\begin{remark}\label{I-jacobson}
  Note that if $w$ is a $v$-gauge on $A$, then $A_0=R_w/I_w$ is a
  finite-dimensional semisimple $F_v$-algebra which is not simple in general,
  cf. \cite[Prop.~2.1. and \S ff.]{T-W-2010}. It follows that $I_w$ is the
  Jacobson radical of $R_w$, cf. \cite[p.~1685]{F-W}.
\end{remark}

If $w$ is a $\s$-invariant surmultiplicative $v$-value function, $\s$ induces
an involution on $A_0$, which we denote by $\s_0$, and, if $\s_0$ is
anisotropic, \cite[Prop.~1.1 (b)$\Rightarrow$(a)]{T-W-2011} shows that $w$ is
$\s$-special.

For future use we record the following two results.

\begin{lemma}\label{power-of-2}
  Let $w$ be a $\s$-special $v$-value function on $A$, where $v$ is any
  valuation on $F$. Let $a \in \Sym(A,\s)$ and $r\in \N$. Then $w(a) > \alpha$
  if and only if $w(a^{2^r}) > 2^r \alpha$.
\end{lemma}

\begin{proof}
  By induction on $r$. Since $w(a^{2^{r+1}}) = w(\s(a^{2^r})a^{2^r}) =
  2w(a^{2^r})$, we have $w(a^{2^{r+1}}) > 2^{r+1}\alpha$ if and only if
  $w(a^{2^r}) > 2^r \alpha$.
\end{proof}

\begin{prop}\label{semisimple}
  Let $w$ be a $\s$-special surmultiplicative $v$-value function on $A$, where
  $v$ is any valuation on $F$. Then $\gr_w(A)$ is semisimple. 
\end{prop}

\begin{proof}
  Since $w$ is surmultiplicative and $\s$-special, we have $w \circ \s = w$,
  cf. Remark~\ref{spec-aniso}. Thus $\s$ induces a grade-preserving involution
  $\wt\s$ on $\gr_w(A)$. Let $I$ be a homogeneous nilpotent $2$-sided ideal of
  $\gr_w(A)$, and let $a \in I_\alpha$ for some $\alpha \in \Gamma$, so $a = b
  + A_{>\alpha}$ for some $b \in A_{\ge \alpha}$. If $w(b) > \alpha$ then $a =
  0$. If $w(b) = \alpha$ then $w(\s(b)b) = 2\alpha$ and $\wt \s(a)a = \s(b)b +
  A_{>2\alpha}$ in $A_{2\alpha}= A_{\ge 2\alpha}/A_{>2\alpha}$ (by definition
  of the product in the graded ring, cf. \cite[p.~98]{T-W-book}). Since $I$ is
  a 2-sided ideal, $\wt \s(a)a \in I$ and thus there is $r \in \N$ such that
  $(\wt \s(a)a)^{2^r} = 0$. Therefore, $(\s(b)b)^{2^r} + A_{>2\cdot 2^r \alpha}
  = 0$ in $A_{2\cdot 2^r \alpha}$, i.e., $w((\s(b)b)^{2^r}) > 2 \cdot 2^r
  \alpha$. By Lemma~\ref{power-of-2} we get $w(\s(b)b) > 2 \alpha$, so $w(b) >
  \alpha$ and thus $a = 0$.
\end{proof}

\begin{remark}\label{ttame} 
  Let $v$ be a valuation on $F$ such that $\ch F_v=0$, and let $B$ be a finite-dimensional 
  semisimple $F$-algebra. Then every $v$-gauge on $B$ is tame, cf. \cite[Def.~1.5, 
  Cor.~3.6]{T-W-2010}.
\end{remark}

Finally, we recall a special case of \cite[Thm.~6.1]{T-W-2011}:

\begin{prop}\label{inspired}
  Let $B$ be a finite-dimensional simple $F$-algebra with $F$-linear involution
  $\tau$ such that $F=\Sym(B,\tau)\cap Z(B)$, and let $v$ be a valuation on $F$
  such that $\ch F_v=0$ and with Henselization $(F^h, v^h)$. Then:
  \begin{enumerate}[$(1)$]
    \item If $\tau\ox \id_{F^h}$ is anisotropic then there is a $\tau$-special 
    $v$-gauge on $B$.
    \item Let $y$ be a $\tau$-special $v$-gauge on $B$. 
    Then $y$ is the unique $\tau$-invariant $v$-gauge on $B$ and is tame.
  \end{enumerate}
\end{prop}

\begin{proof}
  Note that $\tau\ox \id_{F^h}$ is also anisotropic in (2), as observed at the
  start of the proof of \cite[Thm.~6.1]{T-W-2011}, so that $\tau\ox \id_{F^h}$
  is anisotropic in (1) and (2). It follows that the centre $L$ of the
  $F^h$-algebra $B\ox_F F^h$ must be a field, see the third paragraph of the
  proof of \cite[Thm.~6.1]{T-W-2011}. Since $\ch F_v=0$, every algebraic
  extension of $L$ is tamely ramified (cf. \cite[Def.~A.4]{T-W-book}), so that
  the maximal tamely ramified extension of $L$ is its algebraic closure, and
  splits $B\ox_F F^h$. Since furthermore $L$ is tame over $F^h$, $B\ox_F F^h$
  is tame over $F^h$ (cf. \cite[Start of \S2]{T-W-2011}), and so the hypothesis
  of \cite[Thm.~6.1]{T-W-2011} is satisfied.
\end{proof}


\section{Positive cones}

In \cite{A-U-pos} we defined positive cones on finite-dimensional simple
$F$-algebras $A$ with involution $\s$, where $F = Z(A) \cap \Sym(A,\s)$. (Note
that such algebras were simply called ``$F$-algebras with involution'' in
\cite{A-U-pos}.) More general algebras, specifically finite-dimensional
semisimple $F$-algebras $A$ with involution $\s$ where $F\subseteq Z(A) \cap
\Sym(A,\s)$, will naturally appear as residue algebras of gauges, and motivate
the following relaxation of the hypotheses on $A$ and $\s$:

\begin{defi}\label{def-preordering} 
  Let $F$ be a field and $A$ an $F$-algebra. Let $\s$ be an $F$-linear
  involution on $A$. Let $P\in X_F$. A \emph{prepositive cone $\CP$ on $(A,\s)$
  over $P$} is a subset $\CP$ of $\Sym(A,\s)$ such that
  \begin{enumerate}[(P1)]
    \item $0\in \CP$;
    \item $\CP + \CP \subseteq \CP$; 
    \item $\s(a) \cdot \CP \cdot a \subseteq \CP$ for all $a \in A$;
    \item $\CP_F := \{u \in F \mid u\CP \subseteq \CP\}$ is equal to $P$ (we  
    say that $P$ is the 
      \emph{ordering
      associated to} $\CP$);
    \item $\CP \cap -\CP = \{0\}$ (we say that $\CP$ is \emph{proper}).
  \end{enumerate}
  A prepositive cone that is maximal with respect to inclusion is called a \emph{positive cone}. 
\end{defi}

We define $\Nil[A,\s]$ to be the set of all $P\in X_F$ such that there is no
positive cone on $(A,\s)$ over $P$, cf. \cite[\S2 and Prop.~6.6]{A-U-pos}. Note
that $\Nil[A,\s]$ only depends on the Brauer class of $A$ and the type of $\s$.

Note that a prepositive cone $\CP$ induces a partial ordering $\leq_\CP$ on $A$
defined by $a \leq_\CP b$ if and only if $b-a\in \CP$. Moreover, since $\CP$ is
closed under sums, we have $a \leq_\CP b$ and $c \leq_\CP d$ implies $a+c
\leq_\CP b+d$. As usual, we write $a<_\CP b$ for $a \leq_\CP b$ and $a\not=b$.
Also note that $\CP \not = \{0\}$ by (P4).

\begin{remark}\label{aboutPC}
  Let $\CP\subseteq A$, $\CP\not= \{0\}$  and suppose $\CP$ satisfies (P5), then to prove that
  $\CP_F=P$ for a given $P\in X_F$, it suffices to prove that $P\subseteq \CP_F$.
\end{remark}

\begin{ex}\label{PC-PSD} 
  Let $E$ be one of $F$, $F(\sqrt{-1})$ or $(-1,-1)_F$ (the Hamilton quaternion
  division algebra over $F$) and let $\bbar$ denote the identity on $F$ or
  conjugation in the remaining cases. Let $P\in X_F$. Recall that $M \in \Sym
  (M_n(E), \bt)$ is positive semidefinite with respect to $P$ if and only if
  $\ovl{x}^t M x \in P$ for all $x\in E^n$ (see Appendix~\ref{sec:quat-mat} for
  the quaternionic case). The only two (pre)positive cones on $(E, \bbar)$ over
  $P$ are $P$ and $-P$, since $\Sym(E,\bbar)=F$. Therefore, by
  \cite[Prop.~4.10]{A-U-pos}, the only two (pre)positive cones on $(M_n(E),
  \bt)$ are the set $\PSD_n(E,P)$ of positive semidefinite matrices with
  respect to $P$ and the set of negative semidefinite matrices with respect to
  $P$.

\end{ex}

\begin{lemma}\label{P-extends}
  Let $A$ be an $F$-algebra with $F$-linear involution $\s$. Let $\CP$
  be a prepositive cone on $(A,\s)$ over $P \in X_F$, and let $L$ be any
  subfield of $Z(A) \cap \Sym(A,\s)$ containing $F$.  Then $P$ extends to $L$.
\end{lemma}

\begin{proof}
  Let
  \[\Sigma := \{\sum_{i=1}^n \ell_i^2 p_i \mid n \in \N, \ \ell_i \in L,\ p_i 
  \in
  P\}.\]
  Then $\Sigma$ is closed under sums and products (since $P$ is), and contains
  $L^2$. If $\Sigma \cap -\Sigma = \{0\}$ then $\Sigma$ is a preordering on $L$
  containing $P$, which proves the result. Assume $\Sigma \cap -\Sigma \not =
  \{0\}$. Then there are $n \in \N$, $\ell_i \in L$ and $p_i \in P$ such that
  $\sum_{i=1}^n \ell_i^2 p_i = 0$ where not all $\ell_i^2 p_i$ are equal to $0$.
  For instance, $\ell_1^2p_1 \not = 0$. Let $a \in \CP \setminus \{0\}$. Then
  \[0 = (\sum_{i=1}^n \ell_i^2 p_i) a = \sum_{i=1}^n p_i \s(\ell_i) a \ell_i.\]
  But $p_i \s(\ell_i) a \ell_i \in \CP$ for every $i$, since $a \in \CP$, and 
  \[-p_1 \s(\ell_1) a \ell_1 = \sum_{i=2}^n \ell_i^2 p_ia,\]
  with $p_1 \s(\ell_1) a \ell_1 \not = 0$. So we have a nonzero element in $\CP 
  \cap
  -\CP$,  contradicting (P5).
\end{proof}

\begin{lemma}\label{lem:simple}
  Let $A$ be a finite-dimensional simple $F$-algebra with $F = \Sym(A,\s) \cap
  Z(A)$. Let $\CP$ be a prepositive cone on $(A,\s)$ over $P \in X_F$ and let
  $L \subseteq F_P$ be a field extension of $F$. Then $A \ox_F L$ is simple and
  $L = \Sym(A \ox_F L, \s \ox \id) \cap Z(A \ox_F L)$.
\end{lemma}

\begin{proof}
  By [3, Prop.~6.6] we have $P \in \widetilde X_F=X_F \sm \Nil[A,\s]$ and by
  [3, Rem.~6.2] we have that $A \ox_F F_P$ is simple. It follows that $A \ox_F
  L$ is simple.

  Clearly $L \subseteq \Sym(A \ox_F L, \s \ox \id) \cap Z(A \ox_F L)$. We
  consider two cases:

  If $F = Z(A)$ then $F \ox_F L = Z(A) \ox_F L$ i.e., $L = Z(A \ox_F L)$ (cf.
  \cite[Chap.~8, Thm.~3.2(i)]{Sch}), and the result follows.

  If $F \subsetneqq Z(A)$. Then $[Z(A):F] = 2$ and thus $[Z(A \ox_F L) : L] =
  2$. In case $Z(A \ox_F L) \not \subseteq \Sym(A \ox_F L, \s \ox \id)$ then
  $[Z(A \ox_F L): \Sym(A \ox_F L, \s \ox \id) \cap Z(A \ox_F L)]$ is also $2$,
  and thus $L = \Sym(A \ox_F L, \s \ox \id) \cap Z(A \ox_F L)$. In case $Z(A
  \ox_F L) \subseteq \Sym(A \ox_F L, \s \ox \id)$ then for every $a \in Z(A)$
  we have $a \ox 1 \in Z(A \ox_F L)$ and thus $\s(a) \ox 1 = a \ox 1$, from
  which it follows that $\s(a) = a$ and thus $Z(A) \subseteq \Sym(A, \s)$, a
  contradiction to $F \not = Z(A)$.
\end{proof}

\begin{prop}\label{ppc-extension}
  Let $A$ be a finite-dimensional $F$-algebra with $F$-linear involution $\s$.
  Let $\CP$ be a prepositive cone on $(A,\s)$ over $P \in X_F$ and let $(L,Q)$
  be an ordered field extension of $(F,P)$. Then $\CP \ox 1 := \{a \ox 1 \mid a
  \in \CP\}$ is contained in a positive cone $\CP'$ on $(A \ox_F L, \s \ox
  \id)$ over $Q$.

 If $\CP$ is a positive cone, then $\CP = \CP' \cap A$ \tu{(}where we
  identify $\CP$ with $\CP\ox 1$ and $A$ with $A\ox 1$\tu{)}; furthermore, if
  $A$ is simple and $F=\Sym(A,\s)\cap Z(A)$, then $\CP'$ is unique.
\end{prop}

\begin{proof}
  Observe that the results in \cite[\S 5.2]{A-U-pos} hold when $A$ is a
  finite-dimensional $F$-algebra with $F$-linear involution $\s$, and where
  Definition~\ref{def-preordering} is used for prepositive cones. Then the
  first statement is \cite[Prop.~5.8]{A-U-pos}.

  Since $\CP' \cap A$ is easily checked to be a prepositive cone on $(A,\s)$
  over $F$, the second statement is clear since $\CP$ is then a positive cone
  included in $\CP' \cap A$.

  The final statement follows from Lemma~\ref{lem:simple} and
  \cite[Thm.~7.5]{A-U-pos}.
\end{proof}

We can now relate positive cones over $F$ as defined in this paper and positive
cones over $Z(A)\cap \Sym(A,\s)$ as defined in \cite{A-U-pos} for $A$
finite-dimensional and simple:

\begin{prop}\label{ppc-equivalences}
  Assume that $A$ is a
  finite-dimensional simple $F$-algebra with $F$-linear involution $\s$.
  Let $ \CP \subseteq
  \Sym(A,\s)$, $P \in X_F$, and $F_\s := \Sym(A,\s) \cap Z(A)$.
  The following are equivalent:
  \begin{enumerate}[$(1)$]
    \item $\CP$ is a positive cone on $(A,\s)$ over $P$.
    \item $\CP$ is a positive cone on $(A,\s)$ over some ordering on $F_\s$ 
    that extends $P$.
  \end{enumerate}
\end{prop}

\begin{proof}
  (1) $\Rightarrow$ (2): We only have to check that $\CP_{F_\s}$ is an ordering
  on $F_\s$. By Proposition~\ref{ppc-extension}, $\CP = \CP' \cap A$ with
  $\CP'$ a positive cone on $(A \ox_F F_P, \s \ox \id)$, and since $P$ extends
  to $F_\s$ (by Lemma~\ref{P-extends}) we can assume that $F_\s \subseteq F_P$.
  The set $\CP_{F_\s}$ is clearly a preordering on $F_\s$, so it remains to
  show that $\CP_{F_\s} \cup - \CP_{F_\s} = F_\s$. Let $u \in F_\s \subseteq
  F_P= (\CP')_{F_P} \cup - (\CP')_{F_P}$. If $u \in (\CP')_{F_P}$ then, for
  every $a \in \CP$, $ua \in \CP' \cap A = \CP$, so $u \in \CP_{F_\s}$.
  Similarly, if $u \in -(\CP')_{F_P}$ we obtain $u \in -\CP_{F_\s}$, which
  proves the result.
  
  (2) $\Rightarrow$ (1): We only have to check property (P4). It is clear that
  $\CP_F$ is a preordering on $F$, and we only have to check that $\CP_F \cup
  -\CP_F = F$. Let $u \in F \subseteq F_\s$. If $u \in \CP_{F_\s}$ then $u\CP
  \subseteq \CP$, so $u \in \CP_F$. Similarly, if $u \in -\CP_{F_\s}$ then $u
  \in -\CP_F$, proving the result.
\end{proof}

\begin{defi}
  For a subset $S$ of $A$ and $P$ in $X_F$, we define  
  \[\CC_P(S):=\{\sum_{i=1}^n u_i \s(x_i) s_i x_i \mid n\in \N, u_i \in P, x_i\in
  A, s_i\in S\},\]
  and 
  \[\CC(S) := \{\sum_{i=1}^n \s(x_i)s_ix_i \mid n \in \N, x_i \in
  A, s_i \in S\}.\]
\end{defi}

\subsection{Positive cones on finite-dimensional 
simple algebras with involution}

In this section let $A$ be a finite-dimensional simple $F$-algebra with
$F$-linear involution $\s$ such that $F=\Sym(A,\s)\cap Z(A)$ and let $\CP$ be a
positive cone on $(A,\s)$ over $P \in X_F$.

\begin{lemma}\label{lem:CP*}
  Let $\CP^\x := \CP\cap A^\x$. Then  $\CP = \CC(\CP^\x)$. 
\end{lemma}
\begin{proof}
  By \cite[Def.~3.11 and Thm.~7.5]{A-U-pos}, $\CP = \CC_P (S)$ for some set $S$
  of invertible elements in $A$.  Therefore, $\CP$ is contained in $\CC_P
  (\CP^\x)$, and contains it by (P2), (P3) and (P4), so that $\CP = \CC_P(\CP^\x)$.

  But, for $u \in P \setminus
  \{0\}$, $x \in A$ and $b \in \CP^\x$ we have
  \[u \s(x)bx = \s(x) (ub) x, \text{ with } ub \in \CP^\x,\]
  so that $\CC_{P}(\CP^\x) = \CC(\CP^\x)$.
\end{proof}

Recall that the involution $\s$ is called  \emph{positive at $P$} if the form
$\Trd_A(\s(x)x)$ is positive definite at $P$, cf. \cite[\S 6]{A-U-pos}.

\begin{prop}\label{prop:tired}
  The involution $\s$ is positive at $P$ if and only if $(A\ox_F F_P, \s\ox\id) \cong 
  (M_{n_P}(D_P), \bt)$,  where  
  $n_P \in \N$, $D_P \in \{F_P, F_P(\sqrt{-1}), (-1,-1)_{F_P}\}$, and $\bbar$ denotes the identity 
  on $F_P$ and 
  conjugation in
  the remaining cases.  
\end{prop}

\begin{proof}
  By \cite[Rem.~4.7]{A-U-PS}, $\s$ is positive at $P$ if and only if  $(A\ox_F F_P, \s\ox\id) \cong 
  (M_{n_P}(D_P), 
  \Int(\Phi_P)\circ \bt)$, where $\Phi_P$ is a positive definite matrix over $D_P$. This matrix has 
  a square root 
  $\sqrt{\Phi_P}$, and applying $\Int ( \sqrt{\Phi_P}^{-1})$ gives the desired isomorphism.
\end{proof}

\begin{remark}\mbox{}\label{L}
  \begin{enumerate}[$(1)$]
    \item Recall that if $1\in \CP$, then the involution $\s$ is positive at $P$, 
    cf. \cite[Cor.~7.7]{A-U-pos}.
    
    \item By Proposition~\ref{prop:tired},  if $\s$ is positive at $P$
    there exist extensions $L$ of $F$ (even finite ones), with 
    $L\subseteq F_P$ such that  $(A\ox_F L, \s\ox\id) \cong (M_{n_P} (E_L), \bt)$, where $E_L \in 
    \{L, L(\sqrt{-1}), 
    (-1,-1)_{L} \}$. For any such extension $L$ we fix  an isomorphism 
  $f_L: (A\ox_F L, \s\ox\id) \simtoo (M_{n_P} (E_L), \bt)$ of $L$-algebras with involution.
  \end{enumerate}
 \end{remark}

\begin{defi}
  Assume that $\s$ is positive at $P$.
  If $a\in \Sym(A,\s)$, its \emph{$F_P$-eigenvalues} are defined to be the right 
  eigenvalues  (cf. Appendix~\ref{sec:quat-mat}) of the matrix $f_{F_P}(a\ox 1)\in
  M_{n_P}(D_P)$. Note that they belong to $F_P$ since $f_{F_P}(a\ox 1) \in \Sym (M_{n_P}(D_P), 
  \bt)$.
\end{defi}

\begin{prop}\label{prop:ev} 
  Assume that $1\in \CP$.  Let $a\in\Sym(A,\s)$. 
  Then $a\in\CP$ if and only if all $F_P$-eigenvalues of $a$ are nonnegative. Consequently, $1-a 
  \in \CP$ 
  if and only if all $F_P$-eigenvalues of $a$ are $\leq 1$.
\end{prop}

\begin{proof}
  By Proposition~\ref{ppc-extension}, 
  the positive cone $\CP$ extends to a positive cone $\CP'$ on $(A\ox_F F_P, \s\ox\id)$ over $P'$, 
  the unique
  ordering of $F_P$, which is carried  by $f_{F_P}$ to the positive cone $f_{F_P}(\CP')$ on  
  $(M_{n_P} (D_P), \bt)$ over $P'$ (cf. Proposition~\ref{prop:tired} for the notation). By 
  Example~\ref{PC-PSD}
  and since $1\in \CP$,  $f_{F_P}(\CP')$  must be $\PSD_{n_P}(D_P, P')$. The result follows since 
  $a\in \CP$
  if and only if $f_{F_P}(a\ox 1) \in \PSD_{n_P}(D_P, P')$.
\end{proof}

\begin{prop}\label{hs-simple}
  Assume that $F$ is real closed \tu{(}thus $P$ is its unique ordering\tu{)} and that $1\in \CP$. 
  Then $\CP$ is the set of all
  hermitian squares of $(A,\s)$.
\end{prop}

\begin{proof}
  This follows directly from \cite[Cor.~7.15 and Thm.~7.5]{A-U-pos}. Alternatively, a more explicit 
  proof
  is as follows: By Remark~\ref{L}(1) and Proposition~\ref{prop:tired}, $(A,\s) \cong (M_m(D),
  \bt)$ where $m \in \N$ and $D \in \{F, F(\sqrt{-1}), (-1,-1)_F\}$. By \cite[Thm.~7.5]{A-U-pos} 
  there is a unique positive cone on $(M_m(D),\bt)$ over $P$ containing
  $1$, and by Example~\ref{PC-PSD} this positive cone must be $\PSD_m(D,P)$. The
  result follows since every element of $\PSD_m(D,P)$ is a hermitian square
  (the necessary results for the standard argument in the quaternion case can be found in 
  Appendix~\ref{sec:quat-mat}).
\end{proof}

\subsection{Positive cones on finite-dimensional semisimple algebras with involution}

We assume in this section that $A$ is  a finite-dimensional semisimple $F$-algebra equipped with
an $F$-linear involution $\s$. By \cite[1.2.8]{Knus} we may assume that
\[(A,\s) = (A_1,\s_1) \times \cdots \times (A_r,\s_r) \times (B_1,\tau_1) \times
\cdots \times (B_s, \tau_s),\]
where each $A_i$ is a simple $F$-algebra with $F$-linear involution $\s_i$, and
each $B_j$ is equal to $C_j \times C_j^\op$ where $C_j$ is a simple
$F$-algebra and $\tau_j(x,y)=(y,x)$ is the ($F$-linear) exchange involution. 

For $1 \le i \le r+s$, we denote by $e_i$ the element
$(0,\ldots,0,1,0,\ldots,0)$ of $A$ with  $1$  in the $i$-th position.

\begin{prop}\label{pc-ss}
  Let $\CP$ be a prepositive cone on $(A,\s)$ over $P \in X_F$. Then
  \[\CP = \CP_1 \times \cdots \times \CP_r \times 
  \underbrace{\{0\} \times \cdots \times \{0\}}_{s\text{ times}},\]
  where each $\CP_i$ is either a prepositive cone on $(A_i,\s_i)$ over $P$, or
  $\{0\}$, and at least one of $\CP_1, \ldots, \CP_r$ is not $\{0\}$.
  Furthermore, if $\CP$ is a positive cone,
  then all $\CP_i$ are positive cones.
\end{prop}

\begin{proof}
  For $1\leq i \leq r+s$  let 
  $\CP_i = \s(e_i) \CP e_i$ (identifying $\s(e_i) \CP e_i$ with the corresponding subset of $A_i$).
  Obviously $\CP_i \subseteq \Sym(A_i,\s_i)$ for $1\leq i \leq r$, and
  it is immediate   that $\CP_i$ satisfies (P1), (P2), (P3), and also
  (P5) (since $\CP_i \subseteq \CP$). Furthermore $P \subseteq (\CP_i)_F$ and
  by Remark~\ref{aboutPC} we have either $\CP_i = \{0\}$ or $\CP_i$ is a
  prepositive cone on $(A_i,\s_i)$ over $P$.
  Now let $1\leq i\leq s$. Then $\CP_{r+i}=\{0\}$, since otherwise it would be a prepositive cone on $(B_i,\tau_i)$,
   contradicting
 \cite[Rem.~3.3]{A-U-pos}.
  Since, for $a \in \CP$, we have 
  $a = \bigl(\s(e_1)ae_1 , \cdots ,   \s(e_{r+s})ae_{r+s}\bigr)$
  (identifying $\s(e_i)ae_i$ with the corresponding element of $A_i$), we obtain  
  $\CP=\CP_1 \times \cdots \times \CP_r \times \{0\} \times \cdots \times
  \{0\}$.
  Finally, one of $\CP_1, \ldots, \CP_r$ must be different from $\{0\}$,
  otherwise $\CP = \{0\}$, contradicting property (P4). The final statement of the proposition
  is clear.
\end{proof}

\begin{cor}\label{pc-ss-2}
  Let $\CP$ be a positive cone on $(A,\s)$ with $1 \in \CP$. Then $s=0$ and 
  \[\CP = \CP_1 \times \cdots \times \CP_r,\]
  where each $\CP_i$ is a positive cone on $(A_i,\s_i)$ over $P$ with $1
  \in \CP_i$.

  Furthermore, if $F$ is real closed then $\CP$ is equal to the set of 
  hermitian squares in $(A,\s)$.
\end{cor}

\begin{proof}
  The first part of the statement is a direct consequence of
  Proposition~\ref{pc-ss}. 
  
  For the second part: since $F$ is real closed and $Z(A_i)$ is a field that is algebraic
  over $F$ (since $A$ is finite-dimensional over $F$),  we have $[Z(A_i):F]\leq 2$.
  
  It follows that $F=\Sym(A_i,\s_i) \cap Z(A_i)$. Indeed, if $Z(A_i)=F$, this is clear. So assume that
  $Z(A_i)=F(\sqrt{-1})$. If $\s_i(\sqrt{-1})=\sqrt{-1}$, then $-1\in\CP$, a contradiction. Thus 
  $F(\sqrt{-1})$ is not contained in $\Sym(A_i,\s_i)$ and so
  $\Sym(A_i,\s_i) \cap F(\sqrt{-1})$ is a field of index $2$ in $Z(A_i)$ that contains $F$,
  and so must equal $F$. 
  
  Then, by
  Proposition~\ref{hs-simple}, the elements of each $\CP_i$ are
  hermitian squares, and the result follows.
\end{proof}

\begin{lemma}\label{lem:sszero}
 Let $\CP$ be a positive cone on $(A, \s)$.
  Then:
  \begin{enumerate}[$(1)$]
    \item If $\sum_{i=1}^\ell a_i=0$ with $a_1,\ldots, a_\ell \in \CP$, then $a_i=0$ for all $i$.
    \item If $1\in \CP$, then  $\s(a)a=0$ implies $a=0$ for all $a\in A$.
  \end{enumerate}
\end{lemma}

\begin{proof}
  (1) This follows from  property (P5).

  (2) By Corollary~\ref{pc-ss-2}, $1 \in \CP_i$ for $1\leq i \leq r$ and the result follows from the fact that the forms
    $\qf{1}_{\s_i}$ are anisotropic: indeed, 
  if $\qf{1}_{\s_i}$ were isotropic, then $\ell\x 
  \qf{1}_{\s_i}$ would be universal for some $\ell\in \N$ by \cite[Lemma~2.3]{A-U-pos}, and so
  $D_{(A_i,\s_i)}(\ell\x  \qf{1}_{\s_i})=\Sym(A_i,\s_i)\subseteq \CP_i$,
 contradicting that $\CP_i$ is proper.
\end{proof}

\begin{lemma}\label{unique-mpppc}
  Let $\CP$ be a prepositive cone on $(A,\s)$ over $P \in X_F$ such that
  $1 \in \CP$. Then there is a unique positive cone on $(A,\s)$ over $P$
  containing $1$, which then contains $\CP$.
\end{lemma}
\begin{proof}
  Let $\CQ$ be a positive cone on $(A,\s)$ over $P$ containing $1$. By 
  Proposition~\ref{ppc-extension}, $\CQ = \CQ' \cap A$, where $\CQ'$ is a positive cone on
  $(A\ox_F F_P, \s\ox\id)$, and is thus equal to the set of hermitian squares of 
  $(A\ox_F F_P, \s\ox\id)$ by the final part of Corollary~\ref{pc-ss-2}. Therefore,
  $\CQ$ is the intersection of $A$ with the set of hermitian squares of 
  $(A\ox_F F_P, \s\ox\id)$, so only depends on $A$ and $P$.
\end{proof}

\begin{prop}\label{mpppc-inv}
  Let  $\CP$ be a positive cone on $(A,\s)$ over $P \in
  X_F$ such that $1 \in \CP$. Then
  \[\CP = \CC(\CP^\x).\]
\end{prop}
\begin{proof}
  With the same notation as in  Corollary~\ref{pc-ss-2} we write
 $\CP = \CP_1 \times \cdots \times
  \CP_r$ with $\CP_i$ positive cone on $(A_i, \s_i)$ over $P$. By
  Proposition~\ref{ppc-equivalences}, $\CP_i$ is a positive cone on $(A_i, \s_i)$
  over some ordering $Q_i$ of $Z(A_i) \cap \Sym(A_i, \s_i)$.

  By Lemma~\ref{lem:CP*}, $\CP_i = \CC(\CP_i \cap A_i^\x)$.
  Therefore, if $p = (p_1, \ldots, p_r) \in \CP_1 \times \cdots \times \CP_r$ we
  have, for every $i=1, \ldots, r$,
  \[p_i = \sum_{j=1}^n \s_i(x_{ij}) b_{ij}x_{ij},\]
  for some $n \in \N$ (we can take the same $n$ for all $i$), $x_{ij} \in A_i$ and
  $b_{ij} \in \CP_i \cap A_i^\x$. Consequently
  \[p = \sum_{j=1}^n \s((x_{1j}, \ldots, x_{rj})) (b_{1j}, \ldots, b_{rj})
  (x_{1j}, \ldots, x_{rj}),\]
  with $(b_{1j}, \ldots, b_{rj}) \in \CP \cap A^\x$, proving the result.
\end{proof}

\section{Gauges from positive cones}
\label{sec:gauges-pc}

In this section we construct  gauges from positive cones in a natural way, inspired by the classical case of valuations
and orderings.

Let $A$ be an $F$-algebra, $\s$ an $F$-linear involution on $A$, $\CP$  a prepositive cone on 
$(A,\s)$ over $P\in X_F$, 
and $\kk$ a subfield of $F$.
Following Holland \cite[\S4]{Holland-1980}, we 
define the following:
\begin{align*}
R_{{\kk},\CP} &:= \{a \in A \mid \exists m \in {\kk}\cap P \ \s(a)a \le_{\CP} m\},\\
I_{{\kk},\CP} &:= \{a \in A \mid \forall \ve \in {\kk}^\x\cap P \ \s(a)a \le_{\CP} \ve\}.
\end{align*}

Our general strategy is as follows: we first construct a gauge corresponding to
$R_{\kk, \CP}$ for some specific algebra with involution, cf.
Section~\ref{sec:index-two}.  We then address the general case by first reducing to the special
case via a scalar extension,
and then by showing that the restriction of the gauge thus obtained is still a gauge.

Lemmas~\ref{basic} and \ref{basic2} below are reformulations of Holland's results in our context,
using the notation $n(a)$ for $\s(a)a$.

\begin{lemma}\label{basic}
Assume that $1\in \CP$ and
let $a,b \in A$. Then
  \begin{enumerate}[$(1)$]
    \item $n(a) \geq_\CP 0$;
    \item $n(a+b) \leq_\CP 2(n(a)+n(b))$;
    \item $n(ab) = \s(b)n(a)b$;
    \item if $a$ is invertible, then $n(a^{-1}) = n(\s(a))^{-1}$.
  \end{enumerate}
 \end{lemma}

\begin{proof}
(1) Since $1 \in \CP$, we have       $\s(a)a \in \CP$.
 
 (2) It suffices to note that $2(n(a)+n(b)) -  n(a+b) = n(a-b) \in \CP$ as $1\in \CP$. 
 
 (3) and (4) are direct.
 \end{proof}

\begin{lemma}\label{basic2}
  Assume  that $1\in \CP$. Then
  $R_{{\kk},\CP}$ is a subring of $A$ and $I_{{\kk},\CP}$ is an ideal of $R_{{\kk},\CP}$. 
\end{lemma}

\begin{proof}
  By definition, $R_{\kk,\CP}$ contains $1$ and 
  is closed under multiplication by $-1$. It is also clear from 
  Lemma~\ref{basic}(2) that $R_{\kk,\CP}$ and $I_{\kk,\CP}$ are closed under sums.
  Note that for $r,s \in \kk^\x \cap P$ and $a,b\in A$, if (i) $r-n(a) \in\CP$
  and $s-n(b)\in \CP$, then (ii) $rs-n(ab)=r(s-n(b))+\s(b)(r-n(a))b \in \CP$. So, if
  $a,b\in R_{\kk,\CP}$, there exist $r,s$ satisfying (i); then (ii) shows $ab\in R_{\kk,\CP}$.
  Likewise, if $a\in R_{\kk,\CP}$ and $b\in I_{\kk,\CP}$ then (i) holds for some $r$ and all $s$,
  so (ii) shows $ab\in I_{\kk,\CP}$. 
\end{proof}

\begin{remark}\label{oneinp}\mbox{}
  \begin{enumerate}[$(1)$] 
    \item  The hypothesis $1\in \CP$ in Lemma~\ref{basic2} is necessary since, by definition of $I_{{\kk},\CP}$, to have 
    $0 \in I_{{\kk},\CP}$ we must have $1\in \CP$.

    \item If $1\in \CP$, we have $\CP\cap F=P$ (cf. \cite[Prop.~3.8(1)]{A-U-pos}, which also holds if $(A,\s)$ 
    is an $F$-algebra with $F$-linear involution), and 
     it follows from the definition  of $R_{{\kk},\CP}$ that $R_{{\kk},\CP} \cap F = R_{{\kk},P}$
   and $I_{{\kk},\CP} \cap F = I_{{\kk},P}$.

 \end{enumerate}  
\end{remark}

\begin{lemma}\label{u-invert}
  Assume  that $1\in \CP$.
  Let $u \in A$ be unitary, i.e., $\s(u)u= 1$. Then $u \in R_{{\kk},\CP}^\x$.
\end{lemma}

\begin{proof}
  We have $\s(u)u=1 \in \CP$, so $\s(u)u \leq_\CP 1$
  and  thus $u \in R_{{\kk},\CP}$. Since $u^{-1}$ is also unitary, we get $u \in R_{{\kk},\CP}^\x$.
\end{proof}

\subsection{Special case with index $\leq 2$ and conjugate transposition}\label{sec:index-two}
Let $E$ be one of $F$, $C:=F(\sqrt{-1})$ or $H:=(-1,-1)_F$ and let $(A,\s)= (M_n(E), \bt)$,
where $\bbar$ denotes the identity on $F$ or conjugation in the remaining cases. 
Let $n_E(x):=\bar{x}x$ for $x\in E$.
Let $P\in X_F$ and let
$v:F\to \Gamma_v\cup\{\infty\}$ be a valuation on $F$, compatible with $P$.

\begin{lemma}\label{valonfch}
  The function $v_E$ defined by 
  \begin{equation}\label{eq:val}
    v_E(x):=\tfrac{1}{2} v(n_E(x))
  \end{equation}
  is a valuation on $E$ and a tame 
  $v$-gauge on $E$. It is the unique extension of $v$ to $E$. 
\end{lemma}

\begin{proof}
  Let $F^h$ be a Henselization of $F$ with respect to $v$. Since we can take $F^h\subseteq F_P$,
  $H\ox_F F^h$ is a division algebra. Therefore, $v$ extends to a valuation on $H$ by 
  \cite[Thm.~2.3]{Wadsworth-2002}. By \cite[Thm.~2.1 and (2.7)]{Wadsworth-2002} 
  this extension is unique
  and equals $v_H$; furthermore, $v$ then extends uniquely to $C$ and this extension equals $v_C$.
  The fact that $v_E$ is a gauge follows when $E=C$ from
  \cite[Cor.~1.10]{T-W-2010}, and when $E=H$ from \cite[Prop.~1.12]{T-W-2010},
  and it is tame by Remark~\ref{ttame}.
\end{proof}

\begin{lemma}\label{well-known}
  For all 
  $x_1,\ldots, x_r \in E\sm\{0\}$ we have
  \[v_{E}(n_E(x_1) + \cdots + n_E(x_r)) = 2\min\{v_E(x_1), \ldots,  v_E(x_r)\}.\]
\end{lemma}

\begin{proof}
  Since each $n_E(x_i) \in P$ and $v$ is compatible with $P$, we  have
    \begin{align*}
      v_E(n_E(x_1) + \cdots + n_E(x_r)) &= v(n_E(x_1) + \cdots + n_E(x_r))\\
      &=\min\{v(n_E(x_1)), \ldots,  v(n_E(x_r))\}\\
      &=2\min\{v_E(x_1), \ldots,  v_E(x_r)\}.\qedhere
    \end{align*}
\end{proof}

We define a function $w:M_n(E)\to \Gamma_{v_E}\cup\{\infty\}$ as follows, for 
$a=(a_{ij})_{i,j}\in M_n(E)$,
\[
  w(a):= \min_{i,j} \{v_E(a_{ij})\}.
\]

\begin{remark}\label{matrix}
  By definition of $w$, $R_w = M_n(R_{v_E})$ and $I_w =
  M_n(I_{v_E})$. It follows that $R_w/I_w = M_n(R_{v_E}/I_{v_E})$ is simple.
\end{remark}

\begin{lemma}\label{wbtspecial}
  $w$ is a $\bt$-special $v$-gauge on $M_n(E)$. It is the unique
  $\bt$-invariant $v$-gauge on $M_n(E)$.
\end{lemma}

\begin{proof}
  Straightforward computations show that $w$ is a surmultiplicative
  $v$-value function.
  It is also a $v$-norm: let $\{e_{ij}\}$ be the standard $F$-basis of $M_n(F)$ and let
  $\{q_\ell\}$  be a splitting basis for the $v$-norm $v_E$. Then a direct verification
  shows that $\{q_\ell e_{ij} \}$ is a splitting basis for the $v$-value function $w$.

  As observed in Remark~\ref{matrix}, $(M_n(E))_{0}=R_w/I_w$ is simple.
  Furthermore, since
  $\Gamma_w=\Gamma_{v_E}$ and $v_E$ is a valuation, it follows that $\Gamma_w^\x
  = \Gamma_w$. Therefore, by \cite[Lemma~2.13]{T-W-book}, $\gr_w(M_n(E))$ is
  semisimple, and therefore $w$ is a $v$-gauge.
  
  To show that $w$ is $\bt$-special, we use \cite[Prop.~1.1]{T-W-2011}. It is
  clearly a $\bt$-invariant gauge. A direct computation yields $(M_n(E)_0,
  (\bt)_0) \cong (M_n(E_0), {(\bbarr)_0}^t)$, where $(E_0, (\bbarr)_0)$ is one of
  $(F_0,\id)$, $(F_0(\sqrt{-1}), \bbar)$, $(-1,-1)_{F_0}, \bbar)$. In all three
  cases the involution $(\bt)_0$ is anisotropic
  (since hermitian squares in $(E_0, (\bbarr)_0)$
  are sums of squares in $F_0$ in all 
  three cases and $F_0$ is formally real), 
  from which it follows that $\wt{(\bt)}$ is anisotropic by 
  \cite[Cor.~2.3 and Rem.~2.5]{T-W-2011}. Thus $w$ is $\bt$-special by \cite[Prop.~1.1]{T-W-2011}.

  The uniqueness follows from Proposition~\ref{inspired}.
\end{proof}

\begin{lemma}\label{wadcom0}
  Let $g$ be a $\bt$-special surmultiplicative value function on $M_n(E)$, and
  let $u\in M_n(E)$ be unitary, i.e, $\bar{u}^tu=1$, then $g(u)=g(u^{-1})=0$ and
  $g(uau^{-1})=g(a)$ for any $a\in M_n(E)$. 
\end{lemma}

\begin{proof}
  If $u$ is unitary in $M_n(E)$, then $g(\bar{u}^t u)=0$. Thus, since $g$ is
  $\bt$-special, $g(u)=0$. Similarly, $g(u^{-1})=0$. Using surmultiplicativity,
  it follows that $g(uau^{-1})=g(a)$ for any $a\in M_n(E)$.
\end{proof}

Let $\CP$ be a positive cone on $(A,\s)= (M_n(E), \bt)$ over $P$ such that $1\in \CP$.
Note that  $\CP$ is then equal to $\PSD_n(E,P)$ by Example~\ref{PC-PSD} 
and extends  to the positive cone $\PSD_n(E\ox_F F_P, P')$ on $(A\ox_F F_P, \s\ox\id) \cong 
(M_{n} (E\ox_F F_P), \bt)$, where $P'$ is the unique ordering on $F_P$.
Since $v$ is compatible with $P$, there exists a subfield $\kk$ of $F$ such that
$v=v_{\kk,P}$. Consider the valuation  $v'=v_{\kk,P'}$  that extends $v$ to $F_P$.

The following lemma was suggested to us by A. Wadsworth and simplifies a number of arguments in
the paper.

\begin{lemma}\label{wadcom1}
  Let $b\in\Sym(M_n(E),\bt)$, and let $\lambda_1,\ldots, \lambda_n$ be the
  $F_P$-eigenvalues of $b$. Then
  \begin{enumerate}[$(1)$]
    \item $w(b)=\min\{v'(\lambda_1),\ldots, v'(\lambda_n)\}$;
    \item if $b\in\CP$, then $w(b)\leq \min\{v(\bar{x}^tbx) - v(\bar{x}^t x) \mid
      x\in E^n \sm\{0\} \}$, with equality if $\lambda_1,\ldots, \lambda_n \in F$;
    \item if $b\in\CP$, and if $v'(\lambda_1)=\cdots= v'(\lambda_n)$, then
      $v(\bar{x}^tbx)=w(b)+v(\bar{x}^tx)$ for each $x\in E^n \sm\{0\}$.
  \end{enumerate}
\end{lemma}

\begin{proof} 
  For $a=(a_{ij})_{i,j}\in M_n(E\ox_F F_P)$, let $w'(a):= \min_{i,j}
  \{v'_E(a_{ij})\}$, where $v'_E(x)=\tfrac{1}{2}v'(n_{E\ox F_P}(x))$ is the
  unique extension of $v'$ to $E\ox_F F_P$. Note that $w'|_{M_n(E)}=w$, and that
  $w'$ is a $(\bbar \ox \id)^t$-special $v'$-gauge on $M_n(E\ox_F F_P)$ by
  Lemma~\ref{wbtspecial}. It is also immediate that $w'$ is a $v_E$-value
  function.
  
  Note that since $b$ is symmetric, $b\ox 1 \in M_n(E\ox_F F_P)$ is
  diagonalizable by congruences with eigenvalues in $F_P$ by the Principal Axis
  Theorem, cf. Appendix~\ref{sec:quat-mat} for the quaternionic case. Therefore,
  let $u\in M_n(E\ox_F F_P)$ be unitary such that $\bar{u}^t(b\ox
  1)u=\diag(\lambda_1,\ldots, \lambda_n)$ with $\lambda_1,\ldots, \lambda_n \in
  F_P$. 
  
  We now prove the statements.
  
  (1)  By Lemma~\ref{wadcom0}, 
  \[
    w(b)=w'(b\ox 1)=w'(\diag(\lambda_1,\ldots,
    \lambda_n))= \min\{v'(\lambda_1),\ldots, v'(\lambda_n)\}.
  \]
  
  (2) We first show that $w'(b\ox 1)\leq v'(\bar{x}^t(b\ox 1)x) - v'(\bar{x}^t
  x)$ for every $x\in  (E\ox_F F_P)^n \sm\{0\} $. As observed above, we may assume that
  $b\ox 1=\diag(\lambda_1,\ldots, \lambda_n)$ with $\lambda_1,\ldots, \lambda_n
  \in F_P$.  Note that since $b\ox 1 \in \PSD_n(E\ox_F F_P, P')$,
  $\lambda_1,\ldots, \lambda_n \in P'$ and thus, since $v'$ is compatible with
  $P'$, we have $v'(\bar{x}^t(b\ox 1)x)=\min_i\{v'(\lambda_i)+v'(\bar x_i x_i)  \}$ and
  $v'(\bar{x}^t x)=\min_i\{v'(\bar x_i x_i)  \}$. With this observation, the inequality
  holds if $\min_i\{v'(\bar x_i x_i)\}=0$: it is then equivalent to $\min_i \{v'(\lambda_i)\}
  \le \min_i \{v'(\lambda_i) + v'(\bar x_i x_i)\}$, which is true since
  $v'(\bar x_i x_i) \ge 0$ for all $i$.
  For the general case, let $i_0$ be such that $\min_i \{v'(\bar x_i x_i)\} =
  v'({\bar x_{i_0}}x_{i_0})$.
  Then the inequality is equivalent to  
  \[
    w'\bigl(\ovl{x_{i_0}^{-1}} (b \ox 1) x_{i_0}^{-1}\bigr) \le
    v'\bigl(\ovl{(xx_{i_0}^{-1})}^t(b \ox 1) (xx_{i_0}^{-1})\bigr) -
    v'\bigl(\ovl{(xx_{i_0}^{-1})}^t (xx_{i_0}^{-1})\bigr), 
  \]  
  and we are reduced to the
  previous case. It then follows that $w(b) \le \min_i\{v(\bar{x}^tbx) -
  v(\bar{x}^t x) \mid x\in E^n \sm\{0\} \}$.

  Assume now that $\lambda_1,\ldots, \lambda_n \in F$ and let $y_1,\ldots, y_n \in E^n\sm\{0\}$ 
  be eigenvectors
  corresponding to these right eigenvalues. Then
  $v(\bar{y_i}^tby_i) - v(\bar{y_i}^t y_i) = v(\lambda_i)$
  and the result follows.

  (3) This is a direct computation, using that $v'$ is comptatible with $P'$ and
  that $\lambda_1, \ldots, \lambda_n \in P'$ since $b \in \CP$. If $\gamma =
  v'(\lambda_1) = \cdots = v'(\lambda_n)$, then
  \begin{align*}
    v(\bar{x}^tbx) &= v\bigl(\ovl{\bar{u}^tx}^t (\bar{u}^t(b \ox 1) u) \bar{u}^tx\bigr) \\
      &= v(\bar{y}^t \diag(\lambda_1, \ldots, \lambda_n) y),
        \quad \text{where $y = \bar{u}^tx$} \\
      &= v(\sum_i \bar y_i \lambda_i y_i) \\
      &= v(\sum_i \lambda_i \bar y_i y_i), 
        \quad \text{where the summands are all in $P'$} \\
      &= \gamma + \min \{v'(\bar y_i y_i)\} \\
      &= \gamma + v'(\bar{y}^ty) \\
      &= \gamma + v(\bar{x}^tx). \qedhere
  \end{align*}
\end{proof}

\begin{prop}\label{wadcom2}
  We have
  \begin{align*}
    R_w&=\{a\in M_n(E) \mid \text{the $F_P$-eigenvalues of } \bar{a}^ta \text{ belong
      to } R_{v'}\} =R_{\kk,\CP}\\
    \intertext{and}
    I_w&=\{a\in M_n(E) \mid \text{the $F_P$-eigenvalues of } \bar{a}^ta \text{ belong
      to } I_{v'}\} =I_{\kk,\CP}.
  \end{align*}
\end{prop}

\begin{proof}
  We only check the first statement, the second one is similar. Let $a \in
  M_n(E)$. Then $w(a) \ge 0$ if and only if $w(\bar{a}^ta) \ge 0$ by
  Lemma~\ref{wbtspecial} and, by
  Lemma~\ref{wadcom1}(1), this is equivalent to all the $F_P$-eigenvalues of
  $\bar{a}^ta$ belonging to $R_{v'}$, proving the first equality.

  Observe now that if $b \in \CP$ and $r \in F$ then, using extension of
  scalars to $F_P$ and the fact that $\lambda$ is an $F_P$-eigenvalue of $b$ if
  and only if $r - \lambda$ is an $F_P$-eigenvalue of $r - b$:
  \begin{align*}
    b \le_\CP r & \Leftrightarrow (r - b) \ox 1 \in \PSD_n(E \ox_F F_P, P') \\
      &\Leftrightarrow \text{ all $F_P$-eigenvalues of $r - b$ are in $P'$} \\
      &\Leftrightarrow \lambda \le_{P'} r \text{ for every $F_P$-eigenvalue
        $\lambda$ of $b$}.
  \end{align*}
  Applying this observation to $\bar{a}^ta$ we obtain, for $a \in M_n(E)$: there exists $r \in
  \kk \cap P$ such that $\bar{a}^ta \le_\CP r$ if and only if all the $F_P$-eigenvalues of
  $\bar{a}^ta$ belong to $R_{v'}$, proving the second equality.
\end{proof}

\begin{prop}\label{wadcom2.5}
  $w$ is the unique $v$-gauge on $(M_n(E),\bt)$ with gauge ring $R_{\kk,\CP}$.
\end{prop}
\begin{proof}
  Let $g$ be a $v$-gauge on $(M_n(E),\bt)$ with gauge ring $R_{\kk,\CP}$. Then
  $I_g = J(R_{\kk,\CP}) = I_w$ by Remark~\ref{I-jacobson} and thus $R_g/I_g$ is
  simple, cf. Remark~\ref{matrix}.  In this case, by
  \cite[Prop.~3.48]{T-W-book}, $g$ is completely determined by $R_g =
  R_w$.
\end{proof}

\begin{defi}\label{wadcom3}
  Following Proposition~\ref{wadcom2.5}, we define $w_{\kk,\CP} := w$ on $(M_n(E),\bt)$.
\end{defi}

\subsection{The general case}\label{sec:general}

Assume now that $A$ is a finite-dimensional simple $F$-algebra with $F$-linear involution $\s$ such
that $F=\Sym(A,\s)\cap Z(A)$, and that $\CP$ is a positive cone on $(A,\s)$ over $P\in X_F$ such
that $1\in \CP$.

\begin{remark}\mbox{}

    \begin{enumerate}[$(1)$] 
      \item Recall from \cite[Cor.~7.7]{A-U-pos} that $1 \in
    \CP$ means that the involution $\s$ is positive at $P$.

      \item If the hypothesis $1 \in \CP$ is not satisfied, it is possible to
    regain it, at the cost of changing the involution (and the positive cone):
    Take $a \in \CP^\x$ (it exists by \cite[Lemma~3.6]{A-U-pos}). 
    Then $a^{-1} \CP$
    is a positive cone on $(A, \Int(a^{-1}) \circ \s)$ over $P$.
 \end{enumerate}
\end{remark}

Since $1\in \CP$, $\s$ is positive at $P$ by  Remark~\ref{L}(1).
Let $f_{F_P}: (A\ox_F F_P, \s\ox\id) \simtoo (M_{n_P} (D_P), \bt)$, with $n_P$ and $D_P$ as in
Proposition~\ref{prop:tired} and $f_{F_P}$ as in Remark~\ref{L}(2).
Let $P'$ be the unique ordering on $F_P$ and let $\CP'$ be the unique positive
cone on $(A\ox_F F_P, \s\ox\id)$ over $P'$ that extends $\CP$, cf. Proposition~\ref{ppc-extension},
i.e., $\CP' \cap A = \CP$. Let $\CQ=f_{F_P}(\CP')$ and note that $1\in \CQ$.
As observed in the proof of Proposition~\ref{prop:ev}, $\CQ=\PSD_{n_P}(D_P, P')$.
 
Consider the diagram
\begin{equation}\label{diag-w}
  \begin{aligned}
  \xymatrix{
  (A \ox_F F_P, \s \ox \id) \ar[r]^-{\sim}_-{f_{F_P}} & (M_{n_P}(D_P), \bt)
      \ar[d]^{w_{{\kk},\CQ}} \\
  (A, \s) \ar[u]^\zeta &  \Gamma \cup\{\infty\}
}
\end{aligned}
\end{equation}
where $\zeta$ is the canonical morphism $a \mapsto a \ox 1$, and $w_{{\kk},\CQ}$ is the
gauge defined  on $M_{n_P}(D_P)$ obtained from $\CQ$, cf. Definition~\ref{wadcom3}.  
We define 
\begin{equation}\label{eq:gauge}
  w_{{\kk},\CP} := w_{{\kk},\CQ} \circ f_{F_P} \circ \zeta.
\end{equation}
Thus, $w_{{\kk},\CP}(a) = w_{{\kk},\CQ}(f_{F_P}(a \ox 1))$.

A priori $w_{{\kk},\CP}$ depends on $f_{F_P}$. However, we will show in Theorem~\ref{gauge}
that it is a $\s$-special $v_{{\kk},P}$-gauge  
and  therefore unique. 

\begin{prop}
  $R_{w_{\kk, \CP}} = R_{\kk, \CP}$ and $I_{w_{\kk, \CP}} = I_{\kk, \CP}$.
\end{prop}
\begin{proof}
  Since $\CP'$ is an extension of $\CP$, we have $R_{\kk, \CP'} \cap A = R_{\kk,
  \CP}$ and, since $f_{F_P}(\CP') = \CQ$, we have $f_{F_P}(R_{\kk, \CP'}) =
  R_{\kk, \CQ}$. Therefore, for $a \in A$:
  \begin{align*}
    a\in R_{\kk,\CP} & \lra a\ox 1\in R_{\kk,\CP'}\\ 
      &\lra f_{F_P}(a\ox 1)\in R_{\kk,\CQ}\\ 
      &\lra w_{\kk,\CQ}(f_{F_P}(a\ox 1)) \geq 0 \qquad \text{[by
          Proposition~\ref{wadcom2}]} \\
      &\lra w_{\kk,\CP}(a) \geq 0,
  \end{align*}
  and similarly for $I_{\kk,\CP}$.
\end{proof}

\begin{thm}\label{gauge}
  The map $w_{{\kk},\CP}$ is a $\s$-special $v_{{\kk},P}$-gauge on $A$. It is
  tame, and it is the only $\s$-invariant $v_{\kk, P}$-gauge on $A$.
\end{thm}

\begin{proof}
  Let $v'$ be the valuation $v_{\kk,P'}$ on $F_P$ and note that it extends $v_{\kk,P}$.
  
  By Lemma~\ref{lem:sszero}(2), $\s\ox\id_{F_P}$ is anisotropic since $\CP'$ is
  a positive cone containing~$1$. Let $F^h$ denote a Henselization of $F$ with
  respect to $v_{\kk,P}$.  Since $(F_P,v')$ is Henselian, we may assume that
  $F^h\subseteq F_P$ and obtain in particular that $\s\ox\id_{F^h}$ is
  anisotropic. It then follows from Proposition~\ref{inspired}(1) and Remark~\ref{ttame}
  that there is
  a $\s$-special tame $v_{\kk,P}$-gauge $g$ on $A$.
  Therefore, by \cite[Cor.~1.26]{T-W-2010} and Remark~\ref{ttame}, $g \ox v'$ is
  a (tame) $v'$-gauge on $A \ox_F F_P$, and by \cite[Cor.~1.4]{T-W-2011} it is
  $\s\ox\id_{F_P}$-invariant.  
  Since $w_{\kk, \CQ}$ is a $\bt$-special $v'$-gauge (cf.
  Lemma~\ref{wbtspecial}), $w_{\kk, \CQ} \circ f_{F_P}$ is a $\s \ox
  \id$-special $v'$-gauge on $A \ox_F F_P$. Therefore, by
  Proposition~\ref{inspired}(2), $w_{\kk, \CQ} \circ f_{F_P} = g \ox v'$ and it
  follows that $w_{\kk, \CQ} \circ f_{F_P} \circ \zeta  =g$, i.e., $g=w_{\kk,\CP}$.
  Thus, $w_{\kk,\CP}$ is a $\s$-special (tame) $v_{\kk,P}$-gauge. 
  
  The uniqueness statement follows from Proposition~\ref{inspired}(2).
\end{proof}

We finish this section with  a description of the elements in 
\begin{equation}\label{eq:st}
\st (w_{\kk,\CP}):=\{a\in A^\x \mid
w_{\kk,\CP} (a^{-1}) = - w_{\kk,\CP}(a)\},  
\end{equation}
cf. \cite[p.~709]{T-W-2010} (and also \cite[Lemma~1.3]{T-W-2010} for some
properties of the set $\st(w_{\kk,\CP})$):

\begin{prop}\mbox{}
  \begin{enumerate}[$(1)$]

    \item Let $a \in \Sym(A,\s)^\x$. Then $a \in \st(w_{\kk,\CP})$ if and only if all $F_P$-eigenvalues of $a$ 
    have the same $v_{\kk,P}$-value. 

    \item Let $a \in A^\x$. Then $a \in \st(w_{\kk,\CP})$ if and only if all $F_P$-eigenvalues of $\s(a)a$ 
    have the same $v_{\kk,P}$-value. 
  \end{enumerate}
\end{prop}

\begin{proof}
  We denote the unique ordering on $F_P$ by $P'$, $v_{\kk,P'}$ by $v'$ and $w_{\kk,\CP}$ by $w$. 

(1) 
 Let $\lambda_1, \ldots, \lambda_n$ be the $F_P$-eigenvalues of $a$, ordered
  such that $v'(\lambda_1) \le \cdots \le v'(\lambda_n)$ (so that $w(a) =
  v'(\lambda_1)$ by Lemma~\ref{wadcom1}(1)). 
  
  The $F_P$-eigenvalues of $a^{-1}$ are $\lambda_1^{-1}, \ldots, \lambda_n^{-1}$, 
  and have values
  $-v'(\lambda_n) \le \cdots \le -v'(\lambda_1)$, so that $w(a^{-1}) =
  -v'(\lambda_n)$.  Therefore $w(a) = -w(a^{-1})$ if and only if $v'(\lambda_1) =
  v'(\lambda_n)$ if and only if $v'(\lambda_1) = \cdots = v'(\lambda_n)$.
  
(2) Since $w$ is $\s$-special, we have the equivalences
\begin{align*}
w(a^{-1})=-w(a) &\lra   2w(a^{-1})=-2w(a) \lra w(a^{-1} \s(a^{-1}) )= -w(\s(a) a)\\ 
&\lra w ((\s(a)a)^{-1})=-w(\s(a)a)
\end{align*}
and the result follows from (1) since $\s(a)a$ is symmetric.  
\end{proof}

\subsection{Explicit computation in the special case with index $\leq 2$ and arbitrary 
involution}\label{sec:5.3}
The content of this section is partially inspired by a talk of  J.-P. Tignol at the 2009 conference
``Positivity, Valuations, and Quadratic Forms'' at the University of Konstanz.

Let $(A,\s) = (M_n(E),\ad_{h})$ with $(E,\bbar)$ as in Section~\ref{sec:index-two} and
$h = \qf{e_1,e_2,\ldots, e_n}_{-}$, where $e_1,\ldots, e_n \in F^\x = \Sym(E,\bbar)\sm\{0\}$. 
Let $P\in X_F$ and let $\CP$ be a positive cone on $(A,\s)$ over $P$ and with $1\in \CP$.
Let $v$ be a valuation on $F$, compatible with $P$ (i.e., $v=v_{\kk, P}$ for some subfield $\kk$
of $F$)  and let $v_E$ be the unique extension of $v$ to
$E$, cf. Lemma~\ref{valonfch}.

Let $e = \diag(e_1,\ldots,e_n) \in M_n(F)\subseteq M_n(E)$. Then $\s = \Int(e^{-1}) \circ \bt$. 

\begin{lemma}
  The elements $e_1,\ldots, e_n$ are all in $P$, or are all in $-P$.  
\end{lemma}

\begin{proof}
  Since $1\in \CP$, the involution $\ad_h$ is positive at $P$ by \cite[Cor.~7.7]{A-U-pos}
  and so has maximal signature at $P$ by  \cite[Rem.~4.3]{A-U-PS}. Since $\ad_h=(\bt)_e$
  in the notation of \cite[p.~352]{A-U-PS}, the hermitian form $\qf{e}_{-^t}$ has maximal signature at $P$
  in absolute value by \cite[Prop.~4.4(iv)]{A-U-PS}, i.e., the matrix $e$ is positive definite 
  or
  negative definite at $P$. The statement follows.    
\end{proof}

Since $\ad_h=\ad_{-h}$, we may assume that the $e_i$ are all in $P$. It follows that the matrix $e\ox 1$
has a square root $\sqrt{e\ox 1}$  in  $M_n(E\ox_F F_P)$.

Let $\CP'$ be the unique extension of $\CP$ to 
$(A \ox_F F_P, \s\ox\id)$ over $P'$, the unique ordering of $F_P$, 
cf. Proposition~\ref{ppc-extension}. 
Consider the valuation $v'=v_{\kk,P'}$ that extends $v$ to $F_P$ and let $v'_E$ be the unique extension of $v'$ to $E\ox_F F_P$, cf. Lemma~\ref{valonfch}.

After canonically identifying $(A\ox_F F_P, \s\ox\id)$ with 
$(M_n(E\ox_F F_P), \ad_{h\ox 1})$ and noting that
$\ad_{h\ox 1}=\Int(e^{-1} \ox 1)\circ \bt$, we consider the isomorphism
\[f = \Int(\sqrt{e\ox 1}): (M_n(E\ox_F F_P), \Int(e^{-1} \ox 1) \circ \bt) \simtoo (M_n(E\ox_F F_P), \bt).\]
Then $f(\CP')$ is a positive cone over $P'$ containing $1$, and thus $f(\CP')=\PSD_n(E\ox_F F_P,
P')$, the positive cone of all positive semidefinite $n\x n$-matrices over $E\ox_F F_P$, cf.
Example~\ref{PC-PSD}. 
Following Section~\ref{sec:general}, the gauge $w_{\kk,\CP}$ is defined via the following diagram:
\begin{equation*}
  \begin{aligned}
  \xymatrix{
  (M_n(E\ox_F F_P) , \ad_{h\ox 1}) \ar[r]^-{\sim}_-{f} & (M_{n}(E\ox_F F_P), \bt)
      \ar[d]^{w_{{\kk},f(\CP')}} \\
  (M_n(E), \ad_h) \ar[u]^\zeta \ar[r]_{w_{\kk,\CP}}  &  \Gamma \cup\{\infty\}
}
\end{aligned}
\end{equation*}
where
\[
  w_{{\kk},f(\CP')} \bigl((x_{ij})_{i,j} \bigr) = \min_{i,j} \{v'_E (x_{ij}) \}.
\]
For $a=(a_{ij})_{i,j} \in M_n(E)$ we have
\begin{equation}\label{eq:wkp}
  \begin{aligned}
    w_{\kk,\CP}(a)&= w_{\kk, f(\CP')}(f(a\ox 1))\\
    &=\min_{i,j} \{v'_E (\sqrt{e_i \ox 1} (a_{ij}\ox 1) \sqrt{e_j\ox 1}^{-1})\}\\
    &=\min_{i,j} \{v_{E} (a_{ij}) + \tfrac{1}{2}v(e_i)-\tfrac{1}{2}v(e_j)\}.  
  \end{aligned}
\end{equation}
Therefore
\begin{equation}\label{eq:rkp}
\begin{aligned}
  R_{{\kk},\CP}&=\{ (a_{ij})_{i,j}  \in M_n(E) \mid \forall i,j \quad   v_{E}(a_{ij}) 
  \geq  -  \tfrac{1}{2}v(e_i e_j^{-1})\}\\
  &=:\Bigl( \{v_{E}(x) \geq  -  \tfrac{1}{2}v(e_i e_j^{-1})\}\Bigr)_{i,j}.
\end{aligned}
\end{equation}
Similarly, we obtain
\begin{align*}
  I_{{\kk},\CP}&=\{ (a_{ij})_{i,j}  \in M_n(E) \mid \forall i,j \quad   v_{E}(a_{ij}) 
  >  -  \tfrac{1}{2}v(e_i e_j^{-1})\}\\
  &=:\Bigl( \{v_{E}(x) >  -  \tfrac{1}{2}v(e_i e_j^{-1})\}\Bigr)_{i,j}.
\end{align*}
We conclude that
\[R_{{\kk},\CP}/I_{{\kk},\CP} = \Bigl( \{v_{E}(x) \geq  -  \tfrac{1}{2}v(e_i e_j^{-1})\} 
/  \{v_{E}(x)  >  -  \tfrac{1}{2}v(e_i e_j^{-1})\}   
\Bigr)_{i,j}.\]
Observe that if $\tfrac{1}{2}v(e_i e_j^{-1}) \not \in \Gamma_{v_{E}}$, i.e., if
$v(e_i) \not= v(e_j) \bmod 2\Gamma_{v_{E}}$,  then
\[\{v_{E}(x) \geq  -  \tfrac{1}{2}v(e_i e_j^{-1})\} = \{v_{E}(x) >  -  
\tfrac{1}{2}v(e_i e_j^{-1})\}\]
and so
\begin{equation}\label{eq:zero}
\{v_{E}(x) \geq  -  \tfrac{1}{2}v(e_i e_j^{-1})\} /  \{v_{E}(x) >  -  
\tfrac{1}{2}v(e_i e_j^{-1})\} =\{0\}.
\end{equation}

Larmour \cite[\S 3]{Larmour} defined residue forms of hermitian forms over valued division algebras 
with involution. 
We recall what we need from \cite{Larmour}. 
There are nonzero elements $\pi_1,\ldots, \pi_r \in \Sym (E,\bbar)=F$,  such 
that 
$v_{E}(\pi_1)=v(\pi_1), \ldots, v_{E}(\pi_r)=v(\pi_r)$ 
are all distinct modulo $2\Gamma_{v_{E}}$, and
\[h \simeq h_1 \perp \cdots \perp h_r,\]
where
\[h_i=\qf{u_{i,1} \pi_i,\ldots, u_{i,n_i} \pi_i}_{-}\] 
for units $u_{i,\ell}$ in $R_{v_E}$.
The residue forms of $h$ are then the forms 
\[\wt{h}_i= \qf{u_{i,1} + I_{v_E}, \ldots, u_{i,n_i} + I_{v_E}}_{\vt_i}\]
over $(R_{v_E}/ I_{v_E}, \vt_i)$, where $\vt_i$ is the involution induced by 
$\Int{(\pi_i)} \circ \bbar=\bbar$ (since $\pi_i \in F$).

\begin{prop}\label{semisimple2} 
  We have
  \[
    (A_0,\s_0)\cong (M_{n_1}(E_0), \ad_{\wt{h}_1})\x \cdots \x (M_{n_r}(E_0), \ad_{\wt{h}_r})
  \]
  and
  \[
    \Gamma_{w_{\kk,\CP}} = \sum_{i,j=1}^r \tfrac{1}{2}(v(\pi_i)-v(\pi_j))+\Gamma_{v_E}.
  \]
\end{prop}

\begin{proof}
  The computation of $\Gamma_{w_{\kk,\CP}}$ is immediate from \eqref{eq:wkp}.  
  Going back to the notation $h = \qf{e_1,e_2,\ldots, e_n}_{-}$ and, writing  
  \[h_1 = \qf{e_1,\ldots,e_{n_1}}_{-},\  h_2 = \qf{e_{n_1+1},\ldots,e_{n_1+n_2}}_{-}, \ldots\]
  and
  \[I_1 = \{1,\ldots,n_1\}, \ I_2 = \{n_1+1,\ldots,n_1+n_2\}, \ldots\]
  we have $v_{E}(e_i) = v_E(\pi_\ell) $ for every $i \in I_\ell$.
  Then, using \eqref{eq:zero} and the computation leading to it, we obtain
  \[A_0=R_{{\kk},\CP} / I_{{\kk},\CP} \cong \m{M_1 & 0 & & \\ 0 & M_2 & & \\ & & \ddots & 0 \\ & &
  0 & M_r} \cong M_1 \times \cdots \times M_r,\]
  where, for $\ell=1,\ldots, r$,
    \begin{align*}
      M_\ell &= \Bigl( \{v_{E}(x) \geq  -  \tfrac{1}{2}v(e_i e_j^{-1})\} /  
      \{v_{E}(x) >  - \tfrac{1}{2}v  (e_i e_j^{-1})\}  \Bigr)_{i,j \in I_\ell} \\
      &= \Bigl( \{v_{E}(x) \geq  0\} /  \{v_{E}(x) >  0\}  \Bigr)_{i,j \in I_\ell} \\
      &= M_{n_\ell}(R_{v_E}/I_{v_E})\\
      &= M_{n_\ell}(E_0).
    \end{align*}
  Finally, a direct computation shows that $\s_0|_{M_\ell}=\ad_{\wt{h}_\ell}$ for 
  $\ell=1,\ldots, r$.
\end{proof}

\begin{cor}  
  $h$ has only one residue form if and only if $R_{{\kk},\CP}$ is a Dubrovin valuation ring of $M_n(E)$.
\end{cor}

\begin{proof} 
  Assume that $h$ has only one residue form. Then  the proof of Proposition~\ref{semisimple2} together with
  \eqref{eq:rkp}
  shows that $R_{{\kk},\CP} =M_n(R_{v_E})$, which is a Dubrovin valuation ring of $M_n(E)$
  since $R_{v_E}$ is a valuation ring of $E$, cf. \cite[p.~606--607]{Morandi-1989}. 
  
  Conversely, if $R_{{\kk},\CP}$ is a Dubrovin valuation ring of $M_n(E)$ then, by definition,
   $R_{{\kk},\CP}/I_{{\kk},\CP}$ is simple 
  and so $h$
  has only one residue form by Proposition~\ref{semisimple2}.
\end{proof}

\section{Compatibility of positive cones and $\s$-invariant gauges}
\label{sec:comp}

In this section we let $A$ be a finite-dimensional simple $F$-algebra
with $F$-linear involution $\s$ such that $F=\Sym(A,\s)\cap Z(A)$, we
fix $P \in X_F$ and a positive cone $\CP$ on $(A,\s)$ over $P$ such
that $1\in \CP$. Recall that this implies $\CP\cap F=P$, cf.
Remark~\ref{oneinp}(2). If $S\subseteq T$ are subsets of $A$ we say
that $S$ is \emph{$\CP$-convex in $T$} if for all $s_1, s_2 \in S$
and $t\in T$, $s_1 \leq_\CP t \leq_\CP s_2$ implies $t \in S$.

Let $w$ be a $\s$-invariant $v$-gauge on $A$. As before,
we denote the induced involution on $A_0=R_w/I_w$ by $\s_0$. Recall
that $R_w\cap F=R_v$ and $I_w\cap F=I_v$ since $w|_F=v$ and thus that
$A_0$ is an $F_v$-algebra. Denote the canonical projections $R_v\to
F_v$ and $R_w \to A_0$ by $\pi_v$ and $\pi_w$, respectively. Note
that $\pi_w$ extends $\pi_v$.

Inspired by the classical compatibility conditions between valuations and orderings, we consider the following 
properties:

\begin{enumerate}

\item[(C0)] for all $a,b\in \CP$, $w(a+b)=\min\{w(a), w(b)\}$;
\item[(C1)] for all $a,b\in A$, $0\leq_\CP a \leq _\CP b$ implies $w(b) \leq w(a)$;
\item[(C2)] $R_w$ is $\CP$-convex in $A$;
\item[(C3)] $I_w$ is $\CP$-convex in $A$;
\item[(C4)] $I_w$ is $\CP$-convex in $R_w$;
\item[(C5)] $\pi_w(\CP\cap R_w)$ is a prepositive cone on $(A_0,\s_0)$;
\item[(C6)] $a\in \CP\cap I_w\Rightarrow a <_\CP 1$;
\item[(C7)] $1+ \Sym(I_w,\s) \subseteq \CP$.
\end{enumerate}

It is clear that (C0) and (C1) are equivalent, that (C1) implies (C2) and (C3),
and also that (C3) implies (C4). Using the fact that $w_{\kk, \CP}$ is the
unique $\s$-invariant $v_{\kk,P}$-gauge on $A$ (Theorem~\ref{gauge}), we will
show in Propositions~\ref{c1-c8} and \ref{equivalences} below that properties
(C0), \ldots, (C7) are equivalent. However, first we would like to point out
that the equivalences (C4) $\Leftrightarrow$ (C5) and (C6) $\Leftrightarrow$
(C7) can be obtained in an elementary way, using only the definitions of (C0), \ldots, (C7):

\begin{prop}\label{c4-c5} 
\textup{(C4)} implies \textup{(C5)}. More precisely under the hypothesis \textup{(C4)} we have:
 \begin{enumerate}[$(1)$]
    \item $I_v=I_w \cap F$ is $P$-convex in $R_v=R_w \cap F$, and so  
    $\ovl{P}:= \pi_v(P\cap R_v)$ is an ordering  on $F_v$.
    \item The set $\pi_w(\CP\cap R_w)$  
    is a prepositive cone  on $(A_0,\s_0)$ over $\ovl{P}$.
  \end{enumerate}
\end{prop}

\begin{proof}
(1) Let $a \in R_w \cap F$ and $b \in I_w \cap F$ be such that $0 \le_P a
      \le_P b$. Since $1\in \CP$,  $P = \CP \cap F$, so $0 \le_\CP a
      \le_\CP b$ and by hypothesis we obtain $a \in I_w$, so $a \in I_w \cap F$.

(2) Let $\o{\CP}  :=\pi_w(\CP\cap R_w)$. 
      We obviously have $\o{\CP} \not = \varnothing$, $\o{\CP} + \o{\CP}
      \subseteq \o{\CP}$ and, for $\pi_w( a) \in \o{\CP}$ and $\pi_w(x) \in A_0$,
      $ \s_0\bigl(\pi_w(x)\bigr) \pi_w(a) \pi_w(x) = \pi_w(\s(x)ax) \in \o{\CP}$.

      We check $\o{\CP} \cap -\o{\CP} = \{0\}$. Let $a \in \CP\cap R_w$ be such
      that $\pi_w(a) \in \o{\CP} \cap -\o{\CP}$, i.e., there is $b \in \CP \cap R_w$
      such that $\pi_w(a) = -\pi_w( b)$. Then $a+b \in I_w$. By the assumptions on $a$ and $b$
    we have $0 \le_\CP a \le_\CP a+b$. Therefore, by $\CP$-convexity of $I_w$, 
    we get $a \in I_w$, so $\pi_w( a) =  0$.

      We show that $\o{\CP}$ is over $\ovl{P}$, i.e., that $(\o{\CP})_{F_v} =
      \ovl{ P}$ by showing that $\ovl{ P} \subseteq (\o{\CP})_{F_v}$,
      cf. Remark~\ref{aboutPC} ($\ovl{\CP} \not= \{0\}$ since $\ovl{1}\in \ovl{\CP}$). 
      Let $\pi_w(u) \in \ovl{ P}$ with $u \in P\cap R_v=P \cap R_w$, and let $\pi_w(a)
      \in \o{\CP}$ with $a \in \CP \cap R_w$. Then $ua \in \CP \cap R_w$
      by (P4) and the definition of gauge, so 
      $\pi_w(u a)= \pi_w(u) \pi_w(a) \in \o{\CP}$. 
\end{proof}

Note that we will show in Theorem~\ref{prop:maxl} that $\pi_w(\CP\cap R_w)$ is actually maximal, i.e., is a 
positive cone on $(A_0,\s_0)$ over $\ovl{P}$.

\begin{prop} 
  \textup{(C5)} implies \textup{(C4)}.
\end{prop}

\begin{proof} 
  Let $a \in R_w$ and $b \in I_w$ be such that $0 \le_\CP a \le_\CP b$. So $a \in
  \CP$ and $b = a + c$ for some $c \in \CP \cap R_w$. Then $0 = \pi_w( b) = \pi_w( a) +
  \pi_w( c)$, therefore $\pi_w( a) = - \pi_w( c )\in \pi_w(\CP\cap R_w)  \cap -\pi_w(\CP\cap R_w) = \{0\}$, so $a
  \in I_w$.
\end{proof}

\begin{prop}\label{c5-c6}
  \textup{(C5)} implies that $v$ is compatible with $P$.
\end{prop}

\begin{proof} 
  It is sufficient to show that $\pi_v(P\cap R_v)$ is an ordering on $F_v$, cf. 
  \cite[Thm.~3.6.1(1)]{Marshall}.
  The set $\pi_v(P\cap R_v)$ is closed 
  under sum and product and satisfies $F_v\subseteq \pi_v(P\cap R_v) \cup -\pi_v(P\cap R_v)$. Thus to show that 
  $\pi_v(P\cap R_v)$ is an ordering on $F_v$ we only need to show that it is proper. This follows from the 
  observation that
  $\CP \cap R_v=\CP\cap F\cap R_v = P\cap R_v$ and thus $\pi_v(P\cap R_v)=\pi_w(P\cap R_v)
  \subseteq \pi_w(\CP\cap R_w)$,
  which is proper by the assumption and (P5).
\end{proof}

\begin{prop}\label{c7-c8} 
  \textup{(C6)} and \textup{(C7)} are equivalent. 
\end{prop}

\begin{proof} 
  Assume that (C6) holds, i.e., $a\in \CP\cap I_w\Rightarrow 1-a \in
\CP$ (since $a\not=1$). Let $a\in \Sym(I_w,\s)$. Then $a=\s(a)$ and,
using (P3) and $1\in \CP$, $a^2 \in \CP\cap I_w$. Therefore
$1-a^2\in\CP$. By Proposition~\ref{prop:ev}, all $F_P$-eigenvalues of
$a^2$ are at most $1$ and so all $F_P$-eigenvalues of $a$ are in
$[-1,1]$. Thus $1+a$ only has nonnegative $F_P$-eigenvalues, and so
$1+a \in \CP$ by Proposition~\ref{prop:ev}.

  Conversely, assume that (C7) holds. Let $a\in \CP\cap I_w$. Then 
  $a$ and thus $-a$ are in $\Sym(I_w,\s)$. Therefore $1-a\in \CP$ by the
assumption, i.e., $a\leq_\CP 1$. The result follows since $1\not\in
I_w$.

\end{proof}

We now consider the $v_{\kk,P}$-gauge $w_{\kk,\CP}$, cf. \eqref{eq:gauge} in 
Section~\ref{sec:general}.

\begin{prop}\label{c1-c8} 
   \tu{(C0)}, \ldots, \tu{(C7)} hold for $w_{{\kk}, \CP}$.
\end{prop}

\begin{proof}
  We start by proving (C0), from which (C1) up to (C5) follow. 
  
  We first consider the case where $F=F_P$, $(A,\s)=(M_n(E), \bt)$, and $P'$ is the unique ordering
  on $F_P$ (with $E$ as in Section~\ref{sec:index-two}). Let $\CP'=\PSD_n(E,P')$ and let 
  $a,b\in\CP'$. Then, for every $x \in E^n$, we have $\bar x^t a x$, $\bar x^t b x \in P'$. Hence,
  by the $P'$-convexity of $v':=v_{\kk,P'}$, 
  \[
    v'(\bar x^t (a+b) x) = v'(\bar x^t a x + \bar x^t b x) = \min \{v'(\bar x^t a x), 
    v'(\bar x^t b x) \}.
  \]
  Therefore, by Lemma~\ref{wadcom1}(2) (and since all $F_P$-eigenvalues are trivially 
  in $F_P$)  for $w':=w_{\kk,\CP'}$,
  \[
    w'(a+b) = \min_{x\in E^n \sm \{0\}} \{v'(\bar x^t (a+b) x) - v'(\bar x^t x)  \}
    \leq \min_{x\in E^n \sm \{0\}} \{v'(\bar x^t a x) - v'(\bar x^t x)  \} = w'(a).
  \]
  Likewise, $w'(a+b)\leq w'(b)$. Therefore, $w'(a+b)\leq \min \{w'(a), w'(b) \}$ and the reverse
  inequality holds for any value function. The case for general $(A,\s)$ then follows easily by
  extension of scalars to $F_P$ and using the map $f_{F_P}$, cf.
  diagram~\eqref{diag-w} at the start of Section~\ref{sec:general}.
  
  Finally, we prove (C6). Then (C7) follows by Proposition~\ref{c7-c8}.
  For $a \in \CP\cap I_{w_{{\kk},\CP}}$, $v'(\lambda)>0$ for each $F_P$-eigenvalue $\lambda$ of
  $a$, cf. Lemma~\ref{wadcom1}(1). Therefore, each $F_P$-eigenvalue $1-\lambda$ of $1-a$ satisfies
  $1-\lambda>_{P'}0$ by convexity of $I_{v'}$. Hence, $f_{F_P} ((1-a)\ox 1) \in \CP'$, so
  $1-a\in \CP$.
\end{proof}

The fact that (C7) holds for $w_{{\kk},\CP}$ can be strengthened as follows:

\begin{prop}\label{prop:complicated}
  Let $c \in \CP \cap R_{w_{{\kk},\CP}}$ be such that $\pi_{w_{{\kk},\CP}}(c) \in A_0^\x$, and let $\ve \in
  \Sym(I_{w_{{\kk},\CP}}, \s)$. Then $c+\ve \in \CP$.
\end{prop}

\begin{proof}
  (The idea of this proof is due to A.~Wadsworth.)  
  We use the notation from the start of Section~\ref{sec:general}.
  Since $\CP = A \cap \CP'$, $R_{w_{\kk,\CP}} = A \cap
  R_{w_{\kk,\CP'}}$,  $I_{w_{\kk,\CP}} = A \cap I_{w_{\kk,\CP'}}$ and
  $f_{F_P}$ is an isomorphism of algebras with involution, we may assume that
  $F=F_P$, $P$ is the unique ordering on $F_P$, $(A,\s) = (M_n(E), \bt)$ (with $E$
  as in Section~\ref{sec:index-two}), and $\CP = \PSD_n(E,P)$.  Let $v=v_{\kk,P}$ and 
  $w=w_{\kk,\CP}$.

  At the end of the following argument we use that, since $v$ is compatible with
  $P$, for every $a, b \in F$ with $v(a) > v(b)$, $b >_P 0$ implies $a+b >_P 0$
  (as can easily be checked).
  
  We show $c + \ve \in \PSD_n(E, P)$ by showing that $\bar{x}^t (c +
  \ve) x >_P 0$ for every $x \in E^n \setminus \{0\}$. Observe that, since
  $I_w=\ker(\pi_w)$ is the Jacobson radical of $R_w$ (see
  Remark~\ref{I-jacobson}), $\pi_w(c) \in A_0^\x$ implies $c \in R_w^\x$. If
  $\lambda_1, \ldots, \lambda_n$ are the eigenvalues of $c$, then
  $\lambda_1^{-1}, \ldots, \lambda_n^{-1}$ are the eigenvalues of $c^{-1}$, and
  since both $c$ and $c^{-1}$ are in $R_w \setminus I_w$, Lemma~\ref{wadcom1}(1)
  yields $v(\lambda_i) \ge 0$ and $v(\lambda_i) \le 0$ for every $i$, i.e.,
  $v(\lambda_1) = \cdots = v(\lambda_n) = 0$. In particular, and since $c \in
  \PSD_n(E, P)$, $\lambda_1 >_P 0, \ldots, \lambda_n >_P 0$.

  Let $x \in E^n \setminus \{0\}$. Since $w(\ve) > 0$ we obtain by
  Lemma~\ref{wadcom1}(2) that $v(\bar{x}^t\ve x) \ge w(\ve) +
  v(\bar{x}^tx) > v(\bar{x}^tx)$. But, by Lemma~\ref{wadcom1}(3),
  $v(\bar{x}^tcx) = v(\bar{x}^tx)$, so that $v(\bar{x}^t\ve x) >
  v(\bar{x}^t c x)$. Since $\bar{x}^tcx >_P 0$ and $v$ is compatible with $P$,
  it follows that $\bar{x}^tcx + \bar{x}^t \ve x >_P 0$, i.e.,
  $\bar{x}^t(c+ \ve) x >_P 0$.
\end{proof}

\begin{prop}\label{equivalences}
   \textup{(C0)},\ldots, \textup{(C7)} are equivalent. 
\end{prop}

\begin{proof} 
  Since $R_v=R_w \cap F$, $I_v=I_w\cap F$ and $P=\CP \cap F$ (since
$1\in\CP$), each of the properties \textup{(C0)},\ldots,
\textup{(C7)} implies the corresponding property where $A$ is
replaced by $F$, $w$ is replaced by $v$, and $\CP$ is replaced by
$P$. Therefore, any of these properties implies that $v$ is
compatible with $P$, i.e., $v=v_{{\kk},P}$ for some subfield ${\kk}$
of $F$, cf. \cite[Thm.~7.21]{Prestel84}. By Theorem~\ref{gauge},
$w_{{\kk},\CP}$ is the unique $\s$-invariant 
$v_{{\kk},P}$-gauge on $(A,\s)$ and thus $w=w_{{\kk},\CP}$. The other properties then hold since they hold for
$w=w_{{\kk},\CP}$ by Proposition~\ref{c1-c8}.
\end{proof}

\begin{defi} We say that $w$ (recall that $w$ is assumed 
  $\s$-invariant)
  and $\CP$ are \emph{compatible} if any 
  one of the equivalent properties (C0), \ldots, (C7)
  holds. 
\end{defi}

\begin{prop}
  If $w$ is compatible with $\CP$, then there exists a subfield $\kk$ 
  of $F$ such that $v=v_{\kk,P}$
  and $w=w_{\kk,\CP}$. In particular, $w$ is $\s$-special.  
\end{prop}

\begin{proof}
  As observed in the proof of Proposition~\ref{equivalences}, 
  $v=v_{{\kk},P}$ for some subfield
  ${\kk}$ of $F$. Since $w$ is a $\s$-invariant $v$-gauge on $A$, 
  we have $w= w_{\kk,\CP}$ by Theorem~\ref{gauge}.
\end{proof}

 Assume now that $w$ is compatible with $\CP$.
 Let $Q$ be the ordering on $F_v$ that is induced by $P$ and let
 $\CQ$ be the unique positive cone on $(A_0,\s_0)$ that is over $Q$ and such that $1\in \CQ$. 
 Observe that $\CQ$ exists 
 by property~(C5) and Lemma~\ref{unique-mpppc}. 
 We denote by $P'$ the unique ordering on $F_P$ and by $v'$ the natural extension
 of $v=v_{{\kk},P}$ to $F_P$, given by the convex closure of ${\kk}$ with respect to
 $P'$, i.e., $v'=v_{{\kk},P'}$.
 The following lemma is folklore. We provide a proof for the convenience of the reader.

\begin{lemma} \label{K-S} 
  The residue field $(F_P)_{v'}$ is a real closure of $F_v$ at $Q$.
\end{lemma}

\begin{proof}
  By \cite[Thm.~ 1, p.~66]{K-S-1989}, $(F_P)_{v'}$ is
  real closed. It is also algebraic over $F_v$ (since $F_P$ is algebraic over
  $F$) and it is immediate that
  the ordering induced on $(F_P)_{v'}$ by $P'$ extends the ordering on
  $F_v$ induced by $P$ (i.e., $Q$). Therefore $(F_P)_{v'}$ is a real closure of
  $F_v$ at $Q$. 
\end{proof}

In light of Lemma~\ref{K-S}, we may take $ (F_P)_{v'}$ as a real closure $(F_v)_Q$ of $F_v$ at $Q$
in the considerations below. Furthermore, $P'$ is 
then a lifting of the unique ordering $Q'$ of   $(F_v)_Q$.

\begin{thm}\label{prop:maxl}  
  $\pi_w(\CP \cap R_w)$ is a positive cone on
  $(A_0,\s_0)$ over $Q$ containing $1$.
\end{thm}

\begin{proof}  
  By Proposition~\ref{c4-c5}(2) $\pi_w(\CP \cap R_w)$ is a prepositive cone
  on $(A_0,\s_0)$ over $Q$. Let $\CQ$ be the unique positive cone on $(A_0,\s_0)$ over $Q$
  containing $\pi_w(\CP \cap R_w)$, cf. Lemma~\ref{unique-mpppc}. We will show
  that $\pi_w(\CP \cap R_w) = \CQ$.

  We extend scalars to $F_P$ and note that by Lemma~\ref{lem:simple}, 
  $(A\ox_F F_P, \s\ox\id)$ is a finite-dimensional simple $F_P$-algebra
  with $F_P$-linear involution $\s\ox \id$ such that 
  $F_P=\Sym(A\ox_F F_P, \s\ox\id)\cap Z(A\ox_F F_P)$
  (i.e., an $F_P$-algebra with involution in the terminology of 
  \cite{A-U-pos}). 
  
  Let $\CP'$ be the unique positive cone on $(A\ox_F
  F_P, \s\ox\id)$ over $P'$ containing $\CP$ (see \cite[Prop.~5.8]{A-U-pos} and the classification of
  positive cones \cite[Thm.~7.5]{A-U-pos}). We denote  the
  $v'$-gauge $w_{\kk, \CP'}$ by $w'$.
  
  Consider the diagram
  \begin{equation}\label{diag:proj}
    \begin{aligned}
      \xymatrix@R=1ex{
      (A,\s) \ar[rr]& & (A\ox_F F_P, \s\ox\id) \\
      \rotatebox{90}{$\subseteq$} & & \rotatebox{90}{$\subseteq$} \\
      R_{w}\ar[rr]\ar@{->>}[dddd]^{\pi_w}  & & R_{w'}  \ar@{->>}[dd]^{\pi_{w'}}\\
      & & \\
                &  & ((A \ox_F F_P)_0, (\s \ox \id)_0) \\
                & & \\
      (A_0,\s_0)\ar[rr] & &  (A_0 \ox_{F_v} (F_v)_Q, \s_0 \ox \id)\ar@{^{(}->}[uu]_{\iota}\\
      }
    \end{aligned}
  \end{equation}
  where the horizontal arrows are induced by scalar extension,
  $\pi_{w'}$ denotes the canonical residue map associated to $w'$,  and 
  $\iota$ denotes  the canonical inclusion of $A_0 \ox_{F_v} (F_v)_Q$ in $(A \ox_F F_P)_0$,
  cf. \cite[Prop.~2.12]{T-W-book}, 
  i.e., $\iota: \pi_w(x) \ox \pi_{v'}(y) \mapsto \pi_{w'} (x\ox y)$. 
  Note that the lower part of the diagram is commutative by 
  the definition of $\iota$.
  
  By Proposition~\ref{hs-simple}, $\CP'$ is
  the set of hermitian squares in $(A \ox_F F_P, \s \ox \id )$.   
  We denote by $\CQ'$ the unique positive cone on $(A_0 \ox_{F_v} (F_v)_Q, \s_0 \ox \id)$
  over $Q'$ containing $\CQ$, and thus $1$ (see
  Proposition~\ref{ppc-extension} and Lemma~\ref{unique-mpppc}). 
  By Corollary~\ref{pc-ss-2},
  $\CQ'$ is the set of hermitian squares in $(A_0
  \ox_{F_v} (F_v)_Q, \s_0 \ox \id)$.

  Observe that $\pi_{w'}(\CP' \cap R_{w'})$ is a prepositive cone on $((A
  \ox_F F_P)_0, (\s \ox \id)_0)$ by Proposition~\ref{c1-c8}, and thus that
  $\iota^{-1}\bigl(\pi_{w'}(\CP' \cap R_{w'})\bigr)$ is a prepositive cone over
  $Q'$. Indeed, with reference to Definition~\ref{def-preordering}:
  \begin{itemize}
    \item properties (P1) and (P2) are clear (recall that $1 \in \pi_{w'}(\CP'
      \cap R_{w'})$);
    \item (P3) follows from the fact that $\iota$ is a morphism of algebras with
      involution and $\pi_{w'}(\CP' \cap R_{w'})$ is a prepositive cone by
      property (C5) and Proposition~\ref{c1-c8};
    \item (P5) holds since
      $\pi_{w'}(\CP' \cap R_{w'})$ is a prepositive cone and $\iota$ is
      injective;
    \item (P4) holds if the  ordering associated to
      $\iota^{-1}\bigl(\pi_{w'}(\CP' \cap R_{w'})\bigr)$ on $(F_v)_{Q}$ is
      $Q'=\bigl((F_v)_Q\bigr)^2$. This is the case since it contains $((F_v)_Q)^2$
      by (P3), and nothing more by (P5).
  \end{itemize}
  
  We now show that $\CQ'=  \iota^{-1}\bigl(\pi_{w'}(\CP' \cap R_{w'})\bigr)$.
  Since $\iota^{-1}\bigl(\pi_{w'}(\CP' \cap R_{w'})\bigr)$ is a 
  prepositive cone over $Q'$ and contains $1$, by Lemma~\ref{unique-mpppc} it is
  contained in the unique positive cone over $Q'$ containing $1$,
  i.e., in $\CQ'$. 
        
  For the reverse inclusion, let $x \in \CQ'$.  Then $x = (\s_0 \ox \id)(y) y$
  for some $y \in A_0\ox_{F_v} (F_v)_Q$.   Applying $\iota$ to $x$ and letting
  $z\in R_{w'}$ be such that $\pi_{w'}(z)=\iota(y)$, we have  $\iota(x) = (\s
  \ox \id)_0(\pi_{w'}(z)) \pi_{w'}(z) = \pi_{w'}((\s \ox \id)(z)z)$, which
  belongs to $\pi_{w'} (\CP' \cap R_{w'})$  (since $\CP'$ is the set of
  hermitian squares in $(A\ox_F F_P, \s\ox\id)$).  
  Note that the final equality follows from the definition
  of the involution $(\s\ox\id)_0$.
  
  Finally, we show $\CQ=\pi_w(\CP \cap R_w)$. By definition of $\CQ$ we have
  $\pi_{w}(\CP \cap R_{w}) \subseteq \CQ$.   For the other inclusion we first
  show that $\CQ^\x \subseteq \pi_w(\CP \cap R_w)$. The result will then follow
  since $\CQ =\CC(\CQ^\x) \subseteq \pi_w(\CP \cap R_w)$, where the first
  equality holds by Proposition~\ref{mpppc-inv} and the inclusion holds by definition of
  $\CC$ since $\pi_{w}(\CP \cap R_{w})$ is a prepositive cone.
 
  Let $a = \pi_w(b) \in \CQ^\x$ with $b \in R_w\sm I_w$. 
  We first check that we may assume that $b$ is symmetric: Since $a$ is symmetric,
  $\s_0(\pi_w(b))=\pi_w(b)$, i.e., $\pi_w(\s(b))=\pi_w(b)$. Therefore, $b-\s(b) \in I_w$ and it
  follows that $\pi_w(b)=\pi_w(\tfrac{1}{2}(b+\s(b)))$.

  The lower part of
  diagram~\eqref{diag:proj} then yields diagram~\eqref{diag:proj2}:
        
  \begin{equation}\label{diag:proj2}
    \begin{aligned}
      \xymatrix{
      b \ar@{|->}[rr]\ar@{|->}[dd]^{\pi_w}  & & b\ox 1  \ar@{|->}[d]^{\pi_{w'}}\\
                &  & \pi_{w'}(b\ox 1) \\
      a \ar@{|->}[rr] & &  a\ox 1 \ar@{|->}[u]_{\iota}\\
      }
    \end{aligned}
  \end{equation}
        
  Observe that $a \ox 1 \in \CQ'^\x$ and so $\iota(a \ox 1) = \pi_{w'}(c)$ for
  some $c \in \CP'\cap R_{w'} $ 
  since $\CQ'=  \iota^{-1}\bigl(\pi_{w'}(\CP' \cap R_{w'})\bigr)$, and
  $\pi_{w'}(c)$ is invertible.

  The commutativity of diagram \eqref{diag:proj2} then implies  that
  $\pi_{w'}(b\ox 1) = \pi_{w'}(c)$, so $b \ox 1 = c + \varepsilon$ for some
  $\varepsilon \in \Sym (I_{w'}, \s\ox\id)$. Thus, since $c \in \CP'$ and by
  Proposition~\ref{prop:complicated}, 
  we obtain $b \ox 1 \in \CP'$ and so $b \in \CP$.
\end{proof}

\begin{remark}
Observe that by Theorem~\ref{prop:maxl}, we can replace (C5) by
\begin{enumerate}
  \item[(C5')] $\pi_w(\CP\cap R_w)$ is a positive cone on $(A_0,\s_0)$.
\end{enumerate}
\end{remark}

\section{Strong anisotropy of residually positive forms}\label{sec:aniso}

In this section we collect a number of technical results for use in Section~\ref{sec:B-K}.
Let $A$ be a finite-dimensional semisimple $F$-algebra with $F$-linear involution $\s$.

Let $w:A\to \Gamma \cup \{\infty\}$ be a $v$-gauge on $A$ and denote the canonical projection
$R_w\to A_0$ by $\pi_w$.
For each $n\in \N$ we define
  \begin{equation*} 
    w_n : M_n(A) \rightarrow \Gamma \cup \{\infty\},\  w_n((a_{ij})_{i,j}) := \min_{i,j}\{w(a_{ij})\}.
  \end{equation*}
A straightforward computation shows that $w_n$ is a surmultiplicative $v$-value function.
We denote the canonical projection $R_{w_n}\to (M_n(A))_0$ by $\pi_{w_n}$.

Using the definition of $w_n$, it is easy to check that 
$R_{w_n}=M_n(R_w)$, $I_{w_n}=M_n(I_w)$ and that
the map
   \[\xi: (M_n(A))_0  \to M_n(A_0),\ \pi_{w_n}  ((a_{ij})_{i,j}) \mapsto (\pi_w(a_{ij}))_{i,j}\] 
is an isomorphism  of $F_v$-algebras.

Assume that $w$ is invariant under $\s$. Let $U = \diag(u_1,\ldots,u_n)$ with $u_i \in
 R_v^\x=  R_v \setminus I_v$. Then
  \[(\Int(U) \circ \s^t) \bigl((a_{ij})_{i,j}\bigr) = (u_iu_j^{-1} \s(a_{ji}))_{i,j}\]
and so 
\begin{align*}
    w_n \bigl( (\Int(U) \circ \s^t) \bigl((a_{ij})_{i,j}\bigr) \bigr)&= \min_{i,j} \{w (u_iu_j^{-1} \s(a_{ji}))    \}\\
    &= \min_{i,j}\{w (a_{ji} )+ v(u_i)-v(u_j)\}  \\    
    &=\min_{i,j}\{w(a_{ji})\}\\ 
    &= w_n( (a_{ij})_{i,j}),
  \end{align*} 
i.e., $w_n$ is invariant under $\Int(U) \circ \s^t$. It follows that $\Int(U) \circ \s^t$  induces 
a grade-preserving involution on $\gr_{w_n}(M_n(A))$. 

\begin{lemma}\label{lem:isom}
	The map $\xi$ is an isomorphism from
  $\bigl((M_n(A))_0, (\Int(U) \circ \s^t \bigr)_0)$ to 
  $(M_n(A_0),\Int(\xi\circ\pi_{w_n}( U)) \circ (\s_0)^t)$. 
\end{lemma}

\begin{proof}

  Since $u_1,\ldots,u_n \in R_v   \setminus I_v$, a simple computation shows that
  \[\xi (\pi_{w_n} (   (\Int(U) \circ \s^t)(X)))
  = \Int(\pi_{w_n}(U)) \circ (\s_0)^t(\xi(\pi_{w_n}(X)))
  \] 
  for all $X \in R_{w_n}=M_n(R_w)$, i.e.,
  \[
  \xi  (   (\Int(U) \circ \s^t)_0(\pi_{w_n}(X))) =  \Int(\pi_{w_n}(U)) \circ (\s_0)^t(\xi(\pi_{w_n}(X))) 
  \] 
   for all $X \in R_{w_n}=M_n(R_w)$, proving the result.
\end{proof}

Let $R$ be a semisimple
ring and let $\tau$ be an involution on $R$. We say that $\tau$ is \emph{anisotropic} if there is no
nonzero $x\in R$ such that $\tau(x)x=0$ 
and, for $r\in \N$, that $\tau$ is  \emph{$r$-anisotropic} if the involution $\tau^t$ on $M_r(R)$ is anisotropic.
Furthermore, we say that $\tau$ is \emph{strongly anisotropic} if $\tau$ is $r$-anisotropic for every $r\in \N$.

Let $M$ be a right $R$-module,
let $h:M\x M\to R$ be a nonsingular hermitian form over 
$(R,\tau)$ and let $\ad_h$ be the adjoint involution of $h$
on $\End_R(M)$. Recall that $h(\ad_h(f)(x),y)=h(x,f(y))$ for all $x,y\in M$ and $f\in \End_R(M)$.
We call $h$ \emph{$r$-anisotropic} if $r\x h$
is anisotropic.

\begin{lemma}\label{lem:h_aniso}
	$h$ is anisotropic if and only if $\ad_h$ is anisotropic.
\end{lemma}

\begin{proof}
	The left to right implication is straightforward. We prove the other direction by contraposition.
	Assume that $h$ is isotropic, i.e., assume that there is an $x\in M$, $x\not=0$, such that 
	$h(x,x)=0$. Consider the submodule $xR$ of $M$. Since $R$ is semisimple, $xR$ is a direct
	summand of $M$. Let $f$ be the projection from $M$ to $xR$. Then $\ad_h(f)f=0$. 
	Indeed, for all $y,z\in M$ we have
	\[ h(\ad_h(f)f(y),z)=h(f(y), f(z))=0
	\]
	since $h(x,x)=0$. Since $h$ is nonsingular we must have $\ad_h(f)f(y)=0$ for all $y\in M$.	
\end{proof}

\begin{lemma}\label{lem:endo}
	Let $r\in \N$. There is an isomorphism of rings with involution
\[
	\bigl( \End_R (M^r), \ad_{r\x h} \bigr) \cong \bigl( M_r(\End_R(M)), (\ad_h)^t   \bigr),
\]
where $r\x h:=h\perp\cdots \perp h$ \tu{(}$r$ times\tu{)}.
\end{lemma}

\begin{proof}
	Write $M^r=M_1\x \cdots \x M_r$, where $M_1=\cdots=M_r=M$.  Consider
	the ring isomorphism
	\[
		\mu: \End_R(M^r)\to M_r(\End_R(M)),\ f\mapsto (f_{ij})_{i,j},
		\]
	where  $f|_{M_i}=(f_{1i},\ldots, f_{ri})$ with $f_{ij}:M_j\to M_i$, cf. \cite[Chap.~9, 
	Prop.~2.3]{Grillet}.	
	A lengthy, but straightforward computation then shows that 
	$ (\ad_h)^t\circ \mu=\mu\circ \ad_{r\x h}$.	
\end{proof}

Lemmas~\ref{lem:h_aniso} and \ref{lem:endo} have been proved for general nonsingular 
hermitian forms, but will only be used in Lemma~\ref{lem:sa} for a nonsingular diagonal 
hermitian form, in which case their proofs can be obtained by a simple computation.

\begin{lemma}\label{lem:sa}
   Let $s_1, \ldots, s_n \in  \Sym(R, \tau) \cap R^\x$ and 
  $S = \diag(s_1,\ldots,s_n)$. Then for any $r\in \N$,
      $\Int(S) \circ \tau^t$ is $r$-anisotropic on $M_n(R)$ if and only if the hermitian form $\qf{s_1,\ldots, s_n}_\tau$
      is $r$-anisotropic.
\end{lemma}

\begin{proof}
	Observe that $\Int(S) \circ \tau^t$ is adjoint to the hermitian form $\qf{s_1^{-1},\ldots, s_n^{-1}}_\tau$.
	Therefore, $\Int(S) \circ \tau^t$ is $r$-anisotropic on $M_n(R)$ 
	if and only if $(\Int(S) \circ \tau^t)^t$ is anisotropic on $M_r(M_n(R))$ if and only if
	$\ad_{r\x \qf{s_1^{-1},\ldots, s_n^{-1}}_\tau}$ is anisotropic on $M_{rn}(R)$ by Lemma~\ref{lem:endo}
	if and only if
	$r\x \qf{s_1^{-1},\ldots, s_n^{-1}}_\tau$ is anisotropic by Lemma~\ref{lem:h_aniso}. 
	The result follows since the forms
  $\qf{s_1^{-1},\ldots, s_n^{-1}}_\tau$ and $\qf{s_1,\ldots, s_n}_\tau$ are isometric.
\end{proof}

Recall from Remark~\ref{spec-aniso} that the existence of a $\s$-special gauge implies that $\s$ is anisotropic. 
This observation will be used  for several involutions in what follows.

\begin{prop}\label{prop:special}  
  Let $n\in \N$.
  Then:
   \begin{enumerate}[$(1)$]
     \item $w_n$ is a  $v$-gauge on $M_n(A)$.
     \item If the involution $\s_0$ on $A_0$ is $n$-anisotropic, then $w_n$ is $\s^t$-special.
   \end{enumerate}
\end{prop}

\begin{proof}
(1) Let $\{a_\ell\}$
  be a splitting basis of $A$ with respect to $w$ and let $\{E_{ij}\}$ denote the standard basis of $M_n(F)$. Then
  an easy computation shows that $\{a_\ell E_{ij}\}$ is a splitting basis of $M_n(A)$ with respect to $w_n$,
  thus $w_n$ is a $v$-norm. The canonical map 
  \[(M_n(A))_\gamma\to M_n(A_\gamma),\  (a_{ij})_{i,j}+(M_n(A))_{>\gamma} \mapsto 
  (a_{ij} +  A_{>\gamma} )_{i,j}\]
  induces an isomorphism of the graded algebras $\gr_{w_n}(M_n(A))$ and $M_n(\gr_w(A))$, and since the latter 
  is semisimple 
  (writing an ideal $I$ of $M_n(\gr_w(A))$ as $M_n(J)$ for some ideal $J$ of $\gr_w(A)$, a direct verification
  show that $M_n(\gr_w(A))$    
  does not contain any nontrivial $2$-sided nilpotent homogeneous ideal
  since $ \gr_w(A)$ does not),  the former is also semisimple. We conclude that $w_n$ is a $v$-gauge on $M_n(A)$. 
  
(2)  The indices $0$ below are with respect to the gradings induced by $w$ and $w_n$, as appropriate.
  We follow the arguments in the proof of 
  \cite[Prop.~2.12 (iii)$\Rightarrow$(i)]{Kulsh-2011}  (without the restriction that $\s$ is of the  
  first kind and noting that the assumption that
  $j$ is a norm appears to be missing):   Since $\s_0$ is $n$-anisotropic, 
  $(\s_0)^t$ is anisotropic
  on $M_n(A_0)$. By Lemma~\ref{lem:isom} with $U=I_n$ we have
  $\bigl((M_n(A))_0, (\s^t)_0\bigr) \cong 
  (M_n(A_0), (\s_0)^t)$ and so we obtain that $(\s^t)_0$ is anisotropic on $(M_n(A))_0$. Since $w$ is $\s$-invariant,
  $w_n$ is $\s^t$-invariant and the result follows from 
  \cite[Prop.~1.1 (b)$\Rightarrow$(a)]{T-W-2011}.
\end{proof}

\begin{prop}\label{lem:genius}
  Assume that there is a positive cone $\CQ$ on $(A_0,\s_0)$ over $Q\in X_{F_v}$ with $1\in \CQ$.
  Then for any $\ell\in\N$ and any $u_1,\ldots, u_\ell \in R_v^\x= R_v\sm I_v$ such that 
  $\pi_v(u_i) \in Q$,  the hermitian form $\qf{u_1,\ldots, u_\ell}_\s$ is strongly anisotropic.
\end{prop}

\begin{proof} 
  We first show that for all $\ell\in\N$ and $u_1,\ldots, u_\ell \in R_v\sm I_v$ with $\pi_v(u_i) \in Q$,
  the hermitian  form $\qf{\pi_v(u_1),\ldots, \pi_v(u_\ell)}_{\s_0}$ is  anisotropic over $(A_0,\s_0)$.
  Assume therefore that there exist $x_1,   \ldots, x_\ell \in R_w$ such that 
  $\sum_{i=1}^\ell \pi_v( u_i) \s_0(\pi_w(x_i))  \pi_w( x_i) = 0$. Then, by Lemma~\ref{lem:sszero}, 
  every  $\pi_v( u_i) \s_0(\pi_w(x_i))  \pi_w( x_i) $ is zero, and since $\pi_v(u_i)\not=0$ it follows that 
  $\s_0(\pi_w(x_i))  \pi_w( x_i) =0$ for every $i$. Therefore, $\pi_w(x_i)=0$ for every $i$ 
  by Lemma~\ref{lem:sszero} and since $1\in \CQ$. Since there is no bound on $\ell$, it follows
  that  $\qf{\pi_v(u_1),\ldots, \pi_v(u_\ell)}_{\s_0}$ is strongly anisotropic over $(A_0,\s_0)$.

  We now consider the gauge $w_\ell$ on $M_\ell(A)$.  
  By Lemma~\ref{lem:sa}, and with $S:=\diag(u_1,\ldots, u_\ell)$, the involution
  $ \Int(\xi\circ\pi_{w_\ell}( S)) \circ \s_0^t$ is strongly anisotropic on $M_\ell(A_0)$ and so  
  $(\Int(S) \circ \s^t)_0$ is strongly anisotropic on $(M_\ell(A))_0$
  by Lemma~\ref{lem:isom}.    
  Then by Proposition~\ref{prop:special}(2) with $n=1$ (applied to $(M_\ell(A), \Int(S)\circ \s^t)$
   and $w_\ell$ instead of
  $(A,\s)$ and $w$), 
  the gauge $w_\ell$ is $\Int(S)\circ \s^t$-special,
  and thus $\Int(S)\circ \s^t$ is anisotropic on $M_\ell(A)$ by Remark~\ref{spec-aniso}. 
  Therefore, the form $\qf{u_1,\ldots, u_\ell}_\s$ is anisotropic by Lemma~\ref{lem:sa}, and thus strongly
  anisotropic since there is no bound on $\ell$.
\end{proof}

\begin{lemma}\label{lem:omega}
  Assume that $\Gamma_w = \Gamma_v$.
  Let $\omega \in F^\x$ be such that $v(\omega) \not\in 2\Gamma_v$ and let $v'$
  be the unique extension of $v$ to $F(\sqrt{\omega})$. Then, using the
  \tu{(}tame\tu{)} $v'$-gauge $w\ox v'$ on $A\ox_F F(\sqrt{\omega})$, we have
  \[(A_0,\s_0) \cong ((A\ox_F F(\sqrt{\omega}))_0, (\s \ox \id)_0)\]
  as $F_v$-algebras with involution.
\end{lemma}

\begin{proof}
  By hypothesis $\Gamma_w \cap \Gamma_{v'} = \Gamma_v$, so by
\cite[Prop.~2.12]{T-W-book}, the map \[A_0 \ox_{F_0}
F(\sqrt{\omega})_0 \rightarrow (A \ox_F F(\sqrt{\omega}))_0, \quad
\pi_w(a) \ox \pi_{v'}(z) \mapsto \pi_{w\ox v'}(a \ox z)\] is an
isomorphism of $F_0$-vector spaces. A direct verification using the
definition of the product in the graded algebras and the definition
of the residue involution shows that it is a morphism of algebras
with involution.

 Again, a direct verification shows that the $F_0$-algebras with
involution $(A_0 \ox_{F_0} F(\sqrt{\omega})_0, \s_0 \ox \id)$ and
$(A_0, \s_0)$ are isomorphic via the canonical map, and the result
follows.
\end{proof}

\begin{prop}\label{prop:genius}
  Assume that $\Gamma_w = \Gamma_v$ and that there is a positive cone
$\CQ$ on $(A_0,\s_0)$ over $Q\in X_{F_v}$ with $1\in \CQ$. Let $P \in 
X_F$ be a lifting of $Q$ to $F$ \tu{(}i.e., $\pi_v(P\cap R_v)=Q$\tu{)}
and let $a_1, \ldots, a_\ell \in P \setminus
\{0\}$. Then the hermitian form $\qf{a_1, \ldots, a_\ell}_\s$ is
strongly anisotropic.
\end{prop}

\begin{proof}
  Since $a_1, \ldots, a_\ell \in F$, we can write
  \[\qf{a_1, \ldots, a_\ell}_\s \simeq \rho_1 h_1 \perp \cdots \perp \rho_s
  h_s,\]
  where $v(\rho_1), \ldots, v(\rho_r)$ are all different in $\Gamma_v/2\Gamma_v$
  and $h_i = \qf{u_{i,1}, \ldots, u_{i,k_i}}_\s$ with $v(u_{i,j}) = 0$ for all
  $i,j$.
  Let $\omega_1,\ldots, \omega_r \in F$ be such that $v(\omega_1), \ldots,
  v(\omega_r)$ form a $\Z/2\Z$-basis of the subspace of $\Gamma_v/2\Gamma_v$
  generated by $v(\rho_1), \ldots, v(\rho_s)$. 
  Up to replacing $\omega_i$ by
  $-\omega_i$ we can assume that $\omega_1, \ldots, \omega_r \in P$.

  We consider $L := F(\sqrt{\omega_1}, \ldots, \sqrt{\omega_r})$. Since the
  elements $v(\omega_i)$ are linearly independent in $\Gamma_v/2\Gamma_v$, there
  is only one extension $v'$ of $v$ to $L$ and $F_v = L_{v'}$, cf.
  \cite[Thm.~3.3.4]{EP}. Moreover, $P$ extends to an ordering $P'$ on $L$, and
  $v'$ is compatible with $P'$ (indeed: write $v$ as $v_{\kk,P}$; then
  $v_{\kk,P'}$ extends $v$ to $L$, so is equal to $v'$).
  By definition of $L$, $v'(a_i) \in 2\Gamma_{v'}$ for $i=1,\ldots,
  \ell$, so that $a_i = u_i b_i^2$ for some $u_i, b_i \in L$ with $v'(u_i) = 0$.
  Therefore
  \[\qf{a_1, \ldots, a_\ell}_{\s \ox \id} \simeq \qf{u_1, \ldots, u_\ell}_{\s \ox
  \id}\]
  over $(A \ox_F L, \s \ox \id)$, and $u_i = a_i b_i^{-2} \in P'$, so
  $\pi_{v'}(u_i) \in Q$ for $i=1, \ldots, \ell$.
  
  Furthermore, by applying Lemma~\ref{lem:omega} $r$ times, we see that $((A\ox_F L)_0, (\s
  \ox \id)_0) \cong (A_0, \s_0)$, so that there is a positive cone containing $1$ on $((A \ox_F L)_0, (\s
  \ox \id)_0)$ over $Q$. Since  $A \ox_F L$ is semisimple and $\s \ox \id$ is
  $L$-linear, we can apply Proposition~\ref{lem:genius} with $A\ox_F
  L$, and obtain that $\qf{a_1, \ldots, a_\ell}_{\s \ox \id}$ is strongly
  anisotropic. The result follows.
\end{proof}

\section{Baer-Krull type results for positive cones and gauges}\label{sec:B-K}

Let $v:F\to \Gamma_v\cup\{\infty\}$ be a valuation on $F$.
We recall the following presentation of the classical Baer-Krull theorem, cf. 
\cite[pp. 27--28]{P-D-2001}: 
Let $\Omega := \{\omega_i\}_{i \in I} \subseteq F^\x$ be such that $\{v(\omega_i)\}_{i \in I}$ is a 
$\Z/2\Z$-basis
of $\Gamma_v/2\Gamma_v$ and let $\Omega_\pr$ be the set of all finite products of elements of 
$\Omega$
(including $1$). 
Each $a \in F$ can be written in the form $a = u b^2 \rho$ with $v(u)=0$, $b \in F$ and $\rho \in 
\Omega_\pr$. 

\begin{thm}[Baer-Krull]\label{cBK}
  Let $\eta : \Omega \rightarrow \{-1,1\}$ be any map, 
  extended multiplicatively to $\Omega_\pr$. Then, for $Q \in X_{F_v}$, 
    \[P_{(\eta,Q)}:= \{a = ub^2\rho \in F \mid \pi_v(u) \cdot \eta(\rho) \in Q\}   \]
  is an ordering on $F$ that is a lifting of $Q$ (i.e., $\pi_v(P_{(\eta,Q)} \cap R_v) =
  Q$), and all the liftings of $Q$ are of this form. In this case, $v$ is compatible with
  $P_{(\eta,Q)}$.
\end{thm}

Assume now that $A$ is a finite-dimensional simple $F$-algebra with $F$-linear involution $\s$
and centre $K=Z(A)$
such that
$F=\Sym(A,\s)\cap K$. 
Let $w$ be a $\s$-invariant $v$-gauge on $A$ and let 
$\CQ$ be a positive cone on $(A_0,\s_0)$ over $Q \in X_{F_v}$, such that $1 \in \CQ$.
By Theorem~\ref{cBK}, if $P$ is a lifting of $Q$, then $v$ is compatible with $P$ and  
so $v=v_{\kk, P}$ for some subfield $\kk$ of $F$, cf. \cite[Thm.~7.21]{Prestel84}.

A positive cone $\CP$ on $(A,\s)$ over $P$ is called a \emph{lifting} of $\CQ$ (with respect
to $w$) if $P$ is a lifting of $Q$ with respect to $v$ and
  \[\pi_w({\CP \cap R_w} )= \CQ.\]

We consider the involution trace form 
$$\tas: A\x A \to K, (x,y)\mapsto \Trd_A(\s(x)y),$$ 
which is a symmetric bilinear form over $F$ if $\s$ is of the first kind and a hermitian form over $(K,\iota)$ if $\s$ is of the second kind, where $\iota$ is the nontrivial $F$-automorphism of $K$.
Note that $\tas$ is nondegenerate, cf. \cite[\S11]{BOI}. 
Let $\alpha_1,\ldots, \alpha_m \in F$ be the entries in a diagonalization of $\tas$, where  
$m=\dim_K A$.

\begin{thm}\label{bk2}
  Let $P\in X_F$.
  The following statements are equivalent:
  \begin{enumerate}[$(1)$]
    \item $\CQ$ lifts to a positive cone $\CP$ on $(A,\s)$ over $P$, that contains $1$.
    In this case  $w=w_{\kk, \CP}$ \tu{(}where $v=v_{\kk,\CP}$\tu{)} 
    and in particular is compatible with $\CP$.
    
    \item $\s$ is positive at $P$.
    
    \item The form $\tas$ is definite at $P$.
    
    \item $\alpha_1,\ldots, \alpha_m \in P$ or $\alpha_1,\ldots, \alpha_m \in -P$.    
  \end{enumerate}
  
\end{thm}

\begin{proof}
  The equivalence $(2)\lra (3)$ is \cite[Prop.~4.8]{A-U-PS} with $u=1$, and
  the equivalence $(3)\lra (4)$ is clear.
  
  $(1)\Rightarrow (2)$: This follows from \cite[Cor.~7.7]{A-U-pos}. 
  
  $(2)\Rightarrow (1)$: Let $\kk$ be a subfield of $F$ such that $v=v_{\kk,
  P}$.  By \cite[Cor.~7.7]{A-U-pos}, there exists a positive cone $\CP$ on
  $(A,\s)$ over $P$ such that $1\in \CP$. By Theorem~\ref{gauge} and
  Proposition~\ref{inspired}, $w_{\kk,\CP}$
  is $\s$-special and thus the unique $\s$-invariant $v$-gauge on $A$, so $w=w_{\kk,\CP}$.
  Therefore, $w$ and $\CP$ are compatible by Proposition~\ref{c1-c8}, and we
  conclude with Theorem~\ref{prop:maxl} and Lemma~\ref{unique-mpppc}.  
\end{proof}

The equivalent conditions of Theorem~\ref{bk2} are not always satisfied despite the existence of 
the positive cone $\CQ$, as the 
following example shows. 

\begin{ex}\label{bk2-ex}
  Let $F=\R(\!(x)\!)(\!(y)\!)$ be the iterated 
  Laurent series field in the unknowns $x$ and $y$ and let $v:F\to \Z\x \Z$ be the 
  standard $(x,y)$-adic (Henselian) valuation on $F$ (cf. \cite[\S 3]{Wadsworth-2002}). Note that
  $F$ has exactly four orderings, determined by the signs of $x$ and $y$.
  
  Let $A=(x,y)_F$ be the quaternion division algebra with generators $i$ and $j$ such that
$i^2=x$ and $j^2=y$. Let $\gamma$ be the canonical symplectic involution (quaternion
conjugation) on $A$ and also consider the orthogonal involution $\s=\Int(i)\circ \gamma$ on $A$.
Straightforward computations (or see  \cite[(11.3)(1)]{BOI}) show that
  \[
    T_{(A,\gamma)}\simeq \qf{2,-2x,-2y,2xy} \text{ and }
    T_{(A,\s)}\simeq \qf{2,-2x,2y,-2xy}.
  \]

Since $v$ is Henselian it extends uniquely to a valuation $w$ on $A$ (cf. 
\cite[Thm.~2]{Morandi-1989a}), specifically $w(a)=\tfrac{1}{2} v(\gamma(a)a)$, cf. 
\cite[(2.7)]{Wadsworth-2002}. Since $\ch F_v=0$,
$w$ is a $v$-gauge on $A$ by \cite[Prop.~1.13]{T-W-2010}. 
Observe that any $F$-linear involution $\tau$ on $A$ must be anisotropic since
$A$ is a division algebra. 
With reference to \cite[Ex.~6.8]{A-U-prime} we note that $A_0=F_v=\R$.
  
Let $S$ be the ordering on $F$ where $x<0$ and $y<0$. Then $A\ox_F F_S\cong (-1,-1)_{F_S}$, and so
$S\in\Nil[A,\s]$ and $S\not\in \Nil[A,\gamma]$, 
cf. \cite[Def.~3.7]{A-U-Kneb}. If $P$ is any of the other three orderings
on $F$, then $A\ox_F F_P \cong M_2(F_P)$, and so $P\not\in \Nil[A,\s]$, but $P\in \Nil[A,\gamma]$. 
Therefore, $\Nil[A,\s]=\{S\}$ and $\Nil[A,\gamma]= X_F \sm \{S\}$.

Note that since $A_0=\R$, $\gamma_0=\sigma_0=\id$. Therefore, the only positive cone on 
$(A_0,\gamma_0)=(A_0,\s_0)=(\R, \id)$ that contains $1$ is the unique ordering $Q$ of $\R$. 

Theorem~\ref{bk2} describes the orderings of $F$ for which this positive cone can be lifted to a positive cone that contains $1$. In
case of $(A,\gamma)$ this is precisely $S$. In case of $(A,\sigma)$ it is the unique ordering with
$x<0$ and $y>0$.

Observe that for $(A,\gamma)$ the orderings for which we cannot lift the positive cone $Q$ are
exactly the orderings in $\Nil[A,\gamma]$, but that the corresponding set for $(A,\s)$ contains
both orderings that are in $\Nil[A,\s]$ and orderings that are not. In the latter case observe that
there is a positive cone on $(A,\s)$ over the ordering, but it will not contain $1$, cf.
\cite[Prop.~6.6]{A-U-pos}.
\end{ex}
\medskip

In view of Theorem~\ref{bk2} and Example~\ref{bk2-ex}, if $P \in X_F$ is a
lifting of $Q$, there may not always be a lifting $\CP$ of $\CQ$ over $P$ that contains
$1$. In order to describe when such liftings always exists, we investigate more
closely the behaviour of $\tas$.

We define
\[\lift(\CQ) := \{ P \in X_F \mid \text{ there is a lifting of $\CQ$ over $P$,
that contains $1$}\}.\]

Let $(\prescript{\iota}{}A, \prescript{\iota}{}\s)$  denote the conjugate
algebra with involution of $(A,\s)$, cf. \cite[\S3.B]{BOI}. There is an
isomorphism of $K$-algebras with involution
\begin{equation}\label{eq:AAEnd}
  (A\ox_K \prescript{\iota}{}A, \s\ox \prescript{\iota}{}\s) \simtoo
  (\End_K(A), \ad_{\tas}),
\end{equation} 
cf. \cite[Prop.~11.1]{BOI}.
Clearly, the $v$-gauge $w$ is defined on $\prescript{\iota}{}A$ as well:
$w(\prescript{\iota}{}a):=w(a)$ for all $a\in A$.
Let $v'$ be an extension of $v$ to $K$. Recall that $\Gamma_{v'}=\Gamma_v \cup 
(\tfrac{1}{2} v(d)+\Gamma_v)$, where $K=F(\sqrt{d})$.

\begin{lemma}\label{not-extend}
  Assume that $\s$ is of the second kind, and let $P \in \lift(\CQ)$. Then $P$
  does not extend to $K=Z(A)$.
\end{lemma}
\begin{proof}
  If $P$ extends to $Z(A)$, then $P \in \Nil[A,\s]$ by
  \cite[Prop.~8.4]{A-U-pos}.  Indeed, $\Nil[A,\s]$ only depends on the Brauer
  class of $A$  and the type of $\s$, cf. \cite[Def.~3.7]{A-U-Kneb}. Let $D$ be
  the skewfield part of $A$. Then $D$ also carries an involution of the second
  kind (cf. \cite[Thm.~3.1]{BOI}), which we call $\vt$. Then $\Nil[A,\s] =
  \Nil[D,\vt]$ and we can apply \cite[Prop.~8.4]{A-U-pos} since $Z(A)=Z(D)$.

  This contradicts that $\s$ is positive at $P$ (cf. Theorem~\ref{bk2}) by
  \cite[Thm.~6.8]{A-U-pos}.
\end{proof}

\begin{lemma}\label{tensor-gauge}
  Assume that $\lift(\CQ) \not = \varnothing$. Then
  \begin{enumerate}[$(1)$]
    \item $v'$ is the unique extension of $v$ to $K$;
    \item $w$ is a $v'$-gauge on $A$;
    \item $w\ox w$ is a $\s \ox \pri\s$-invariant $v'$-gauge and $v$-gauge on
      $A\ox_K \pri A$.
  \end{enumerate}
\end{lemma}
\begin{proof}
  (1) The statement is obvious if $\s$ is of the first kind, so we assume that
  $\s$ is of the second kind. If there are two extensions of $v$ to $K$ then, by
  the fundamental inequality of valuation theory, $\Gamma_{v'}=\Gamma_{v}$ and
  $K_{v'}=F_v$. Let $P \in \lift(\CQ)$. Then $P = P_{(\eta,Q)}$ as described in
  Theorem~\ref{cBK} and, again by Theorem~\ref{cBK}, the set
  \[\{ub^2\rho \in K \mid u \in R_{v'} \setminus I_{v'},\ b \in K,\ \rho \in
  \Omega_\pr \text{ such that } \pi_{v'}(u) \cdot \eta(\rho) \in Q\}\]
  defines an ordering on $K$ extending $P$, which is not possible by
  Lemma~\ref{not-extend}.

  (2) This follows from (1) by \cite[Prop.~2.1]{F-W}.

  (3) $w\ox w$ is a $v'$-gauge on $A\ox_K \pri A$ by \cite[Cor.~1.28]{T-W-2010}
  and Remark~\ref{ttame}, and is also a $v$-gauge by \cite[Prop.~2.1]{F-W}. It
  is clearly $\s \ox \pri\s$-invariant.
\end{proof}

In order to keep notation simple,we identify $A\ox_K \pri A$ and $M_m(K)$ and
consider the  gauge $w \ox w$ as being defined on $M_m(K)$.

We now write the form $\tas$ as
\[\tas \simeq \rho_1 h_1 \perp \dots \perp \rho_k h_k,\]
where $\rho_1, \ldots, \rho_k$ are distinct elements of $\Omega_\pr$, and $h_1,
\ldots, h_k$ are diagonal forms over $(K,\iota)$ with coefficients in $R_v \setminus I_v$. Since
$m$ is represented by $\tas$ we can assume that $\rho_1 = 1$, and that the first
coefficient of $h_1$ is $m$.

\begin{lemma}\label{u-in-Q}
  Assume that $\lift(\CQ) \not = \varnothing$.
  For every $i=1, \ldots, k$, either all coefficients of $h_i$ project via
  $\pi_v$ into $Q$ \tu{(}in this case we define $\ve_i:=1$\tu{)}, or all coefficients of
  $h_i$ project via $\pi_v$ into $-Q$ \tu{(}in this case we define $\ve_i := -1$\tu{)}.
\end{lemma}
\begin{proof}
  Let $P \in \lift(\CQ)$. By Theorem~\ref{bk2}, the form $\tas$ is definite at
  $P$. By   \cite[Prop.~4.8 with $u=1$]{A-U-PS}, $\s$ and thus $\s\ox\pri\s$ are positive at $P$.
  Therefore, the involution $\ad_{\tas}$
  is positive at $P$, and by \cite[Cor.~7.7]{A-U-pos} there is a positive cone
  $\CS$ on $(M_m(K), \ad_{\tas})$ that contains $1$. We can therefore define the
  $\ad_{\tas}$-special $v$-gauge $w_{\kk,\CS}$ on $M_m(K)$ (cf.
  Section~\ref{sec:general}; actually, since $w \ox w$ is an
  $\ad_{\tas}$-invariant $v$-gauge, we have $w \ox w = w_{\kk, \CS}$ 
  by
  Theorem~\ref{gauge}), and this $v$-gauge induces a positive cone
  $\CS_0$ on $((M_m(K))_0, (\ad_{\tas})_0)$ over $Q$ that contains $1$ by
  Theorem~\ref{prop:maxl} (or by property (C5)).

  We consider the residue algebra with involution $((M_m(K))_0, (\ad_{\tas})_0)$
  with respect to the gauge $w_{\kk,\CS}$.  By Proposition~\ref{semisimple2},
  $(M_m(K))_0 \cong B_1 \times \cdots \times B_k$, with each $B_i$
  simple, and the residue involution is $(\ad_{\tas})_0 \cong \ad_{\wt h_1} \x
  \cdots \x \ad_{\wt h_k}$.  By Corollary~\ref{pc-ss-2}, $\CS_0 = \CS_1 \times
  \cdots \times \CS_k$ with $\CS_i$ a positive cone on $(B_i, \ad_{\wt h_i})$
  over $Q$, that contains $1$. Therefore, the involution $\ad_{\wt h_i}$ is
  positive at $Q$ (see \cite[Cor.~7.7]{A-U-pos}). By
  \cite[Prop.~4.4(iv)]{A-U-PS} (using the diagonal matrix $u$ that contains the coefficients of
  $\wt h_i$, and noting that $\ad_{\wt h_i}= (\s_0^t)_u$), the  coefficients of $\wt h_i$ are
  either all in $Q$, or all  in $-Q$.
\end{proof}

Observe that $\ve_1=1$ since the first coefficient of $h_1$ is $m$.

\begin{prop}\label{main-BK}
  Assume that $\lift(\CQ) \not = \varnothing$.
  \begin{enumerate}[$(1)$]
    \item $\lift(\CQ)$ is the intersection of the Harrison set $H(\ve_1 \rho_1,
      \ldots, \ve_k \rho_k)$ with the set of all liftings of $Q$ to $F$.
    \item $\Gamma_w \subseteq \Gamma_{w
      \ox w} = \sum_{i,j=1}^k \tfrac{1}{2}( v(\rho_i) -v(\rho_j))  +
      \Gamma_{v'}$.
  \end{enumerate}
\end{prop}
\begin{proof}
  (1) is a direct consequence of Theorem~\ref{bk2}$(4)$ and the definition of
  the elements $\ve_i$ in Lemma~\ref{u-in-Q}. Indeed, a lifting $P$ of $Q$ is in
  $\lift(\CQ)$ if and only if the coefficients of $\tas$ are all in $P$ or all
  in $-P$, if and only if each $\ve_i \rho_i$ is positive at $P$ (or negative at
  $P$, but this case is not possible since $\rho_1=1$ and $\ve_1=1$).
  
  (2) Let $P\in\lift(\CQ)$, then $\tas$ is positive at $P$ by Theorem~\ref{bk2}.
  It follows, using \cite[Cor.~7.7]{A-U-pos} and as already observed in the
  proof of Lemma~\ref{u-in-Q}, that there exists a positive cone $\CS$ on
  $(M_m(K), \ad_{\tas})$ over $P$ that contains $1$.  We are therefore in the
  situation of Section~\ref{sec:5.3} with $(E,\bbar)=(K,\iota)$ and $h=\tas$,
  and so may consider the $\ad_{\tas}$-special $v$-gauge $w_{\kk,\CS}$.  By
  Proposition~\ref{inspired} and Theorem~\ref{gauge}, this is the unique $\ad_{\tas}$-invariant
  $v$-gauge on $\End_K(A)$. Therefore, $w\ox w = w_{\kk,\CS}$ and $\Gamma_{w\ox
  w}=\Gamma_{w_{\kk, \CS}}$. The result then follows by
  Proposition~\ref{semisimple2}, using that $\Gamma_w\subseteq \Gamma_{w\ox w}$
  (see \cite[Thm.~3.21(ii)]{T-W-2010}).
\end{proof}

\begin{lemma}\label{kinds}
  Assume that $\s$ is of the second kind, and that for every lifting $P$ of $Q$
  to $F$, the involution $\s$ is positive at $P$. Then 
  $\Gamma_v = \Gamma_{v'}$.
\end{lemma}
\begin{proof}
  (This elegant argument was suggested to us by A.~Wadsworth.) Write
  $K=F(\sqrt{d})$ for some $d\in F$. An ordering on $F$ extends to $K$ if and
  only if $d$ is positive at this ordering.  By the assumption and
  Lemma~\ref{not-extend}, $d$ must be negative at every lifting of $Q$ to $F$.
  However, if $v(d)$ were not in $2\Gamma_v$, then $d$ would be positive at some
  liftings of $Q$ by the Baer-Krull theorem. Hence, $v(d)$ must be an element of
  $2\Gamma_v$, which implies that $\Gamma_{v'} = \Gamma_{v}$.
\end{proof}

We can now describe when every lifting of $Q$ to $P$ gives rise to a lifting of
the positive cone $\CQ$ that contains $1$.

\begin{thm}\label{wadth}
  The following are equivalent:
  \begin{enumerate}[$(1)$]
    \item For every lifting $P$ of $Q$, there is a lifting $\CP$ of $\CQ$ that contains $1$.
    \item $\Gamma_w = \Gamma_v$.
  \end{enumerate}
  If one of these holds, then the number of liftings of $\CQ$ to $(A,\s)$ that
  contain $1$ is equal to $|\Gamma_v/2\Gamma_v|$.
\end{thm}
\begin{proof}
  (1) $\Rightarrow$ (2): By Proposition~\ref{main-BK}(1) we obtain that the set
  of all liftings of $Q$ to $F$ is included in $H(\ve_1\rho_1, \ldots,
  \ve_k\rho_k)=H(1, \ve_2\rho_2 \ldots, \ve_k\rho_k)$.  By the Baer-Krull
  theorem, this can only happen if $k=1$ (since otherwise there would be a
  lifting of $P$ at which $\ve_2\rho_2$ would be negative). In this case,
  Proposition~\ref{main-BK}(2) gives $\Gamma_w = \Gamma_{v'}$.  This concludes
  the proof if $\s$ is of the first kind. If $\s$ is of the second kind,
  Lemma~\ref{kinds} (using Theorem~\ref{bk2}) gives $\Gamma_{v'} = \Gamma_v$.

  (2) $\Rightarrow$ (1): Let $P$ be a lifting of $Q$ to $F$, determined by the
  function $\eta$, cf. Theorem~\ref{cBK}. We define 
  \[\CP_{(\eta,Q)} := \Bigl\{\sum_{i=1}^r u_i   \rho_i  \s(x_i)x_i \,\Bigl | \Bigr.\, 
  r \in \N,\ u_i \in R^\x_v,\  \rho_i\in \Omega_\pr,\    \pi_v(u_i)\cdot \eta(\rho_i)
  \in Q,   \ x_i \in A \Bigr\}\]
  and observe that $u_i   \rho_i \in P_{(\eta,Q)}$ and $1 \in
  \CP_{(\eta,Q)}$. The set $\CP$ is the
  closure of $\{1\}$ under the operations defining a positive cone over $P$,
  cf. properties (P2), (P3), and (P4) in Definition~\ref{def-preordering}.
  We first check that $\CP_{(\eta,Q)}$ is a prepositive cone on
  $(A,\s)$ over $P_{(\eta,Q)}$:

  Properties (P1), (P2) and (P3) from Definition~\ref{def-preordering} are
  clearly satisfied. Property (P5) follows from Proposition~\ref{prop:genius}
  since the elements $u_i \rho_i$ are in $P_{(\eta,Q)} \setminus \{0\}$.
  
  For (P4), and using that (P5) holds, it suffices to show that $P_{(\eta, Q)}
  \subseteq (\CP_{(\eta,Q)})_F$, cf.  Remark~\ref{aboutPC}.  Thus, let $a\in
  P_{(\eta, Q)}$ and write $a = u b^2 \rho$ with $v(u)=0$, $b \in F$, $\rho \in
  \Omega_\pr$ and $\pi_v(u)\eta(\rho) \in Q$.  Then, following the definition of
  $\CP_{(\eta,Q)}$, 
  \[a(u_i  \rho_i \s(x_i)x_i)=   (u u_i) (\rho\rho_i)  \s(bx_i)bx_i \]
  is again an element of $\CP_{(\eta,Q)}$.

  Let $\CP$ be the unique positive cone on $(A,\s)$ containing $\CP_{(\eta,Q)}$.
  By Theorem~\ref{gauge} and Proposition~\ref{inspired}, $w_{\kk,\CP}$ is the
  unique $\s$-invariant $v$-gauge on $(A,\s)$ and $w=w_{\kk,\CP}$. Therefore,
  $w$ and $\CP$ are compatible by Proposition~\ref{c1-c8}, and we conclude with
  Theorem~\ref{prop:maxl} and Lemma~\ref{unique-mpppc}.
  
  This completes the proof of the equivalences. The final statement of the
  theorem is now clear: In this case, a lifting of $\CQ$ that contains $1$ is
  completely determined by the lifting $P$ of $Q$ (cf.
  Lemma~\ref{unique-mpppc}), and the number of such liftings is determined by
  the Baer-Krull theorem.
\end{proof}

One might wonder if there could be a clear relation between the number of cosets
of $\Gamma_v$ in $\Gamma_w$ and the number of liftings of $\CQ$. This does not
appear to be the case, as the following example shows:

\begin{ex}
  Let $F$ be a field and let $v$ be a real valuation on $F$ 
  such that $\{\rho_1, \rho_2, \rho_3, \rho_4\} 
  \subseteq  \Gamma_v$ is a basis of $\Gamma_v/2\Gamma_v$. Let $A=M_6(F)$ and consider the involutions 
  $\ad_{\vf}$ and $\ad_\psi$ on $A$, where
  \[
    \vf=\qf{1,\rho_1,\rho_2,\rho_3, \rho_4,  \rho_1\rho_2\rho_3\rho_4}\quad \text{ and }\quad 
    \psi=\qf{1,\rho_1,\rho_2,\rho_3, \rho_1\rho_2, \rho_3\rho_4}.
  \]
  Let $P$ be an ordering  on $F$ such that $\rho_1, \rho_2,   \rho_3, \rho_4 \in P$. 
  Since $\ad_\vf$
  and $\ad_\psi$ are positive at $P$, there is a positive cone $\CP_1$ on $(A,\ad_\vf)$ over
  $P$ and a positive cone $\CP_2$ on $(A,\ad_\psi)$ over  $P$ such that both contain $1$, cf.
  \cite[Cor.~7.7]{A-U-pos}.
  Thus, we may consider the $v$-gauges $w_{\Q,\CP_1}$ and $w_{\Q,\CP_2}$. 
  Using Proposition~\ref{semisimple2} a direct, but lengthy, verification gives
  \[
    [\Gamma_{w_{\Q,\CP_1}}:\Gamma_v]=16 \quad \text{ and }\quad [\Gamma_{w_{\Q,\CP_2}}:\Gamma_v]=13.
  \]
  By Theorem~\ref{prop:maxl},  
  the positive cones $\CP_1$ and $\CP_2$ induce positive cones $\CQ_1$ and $\CQ_2$, respectively,
  on the residue algebras with involution of $(A,\ad_\vf)$ with respect to $w_{\Q,\CP_1}$ and 
  $(A,\ad_\psi)$ with respect to $w_{\Q,\CP_2}$, respectively. In fact,
  \[
    (A_0,(\ad_\vf)_0) \cong (A_0,(\ad_\psi)_0) \cong   (F_v\x F_v \x F_v\x F_v\x F_v\x F_v,\id)
  \]
  by Proposition~\ref{semisimple2} and it follows from Corollary~\ref{pc-ss-2} that
  \[
    \CQ_1=\CQ_2= Q\x Q\x Q\x Q\x Q\x Q,
  \]
  where $Q$ is the ordering on $F_v$ induced by $P$.
  
  Observe that $P$ is the only ordering on $F$ for which $\ad_\vf$ is positive. Likewise,
  $P$ is the only ordering on $F$ for which $\ad_\psi$ is positive. Therefore,
  \[
    |\lift(\CQ_1)|= |\lift(\CQ_2)|=1,
  \]
  by Theorem~\ref{bk2}.  
\end{ex}

\appendix

\section{Quaternionic matrices and eigenvalues}\label{sec:quat-mat}

Given an ordering $P$ on $F$, several of our proofs consist of a reduction to
the case of matrices over $F_P$, $F_P(\sqrt{-1})$ and $(-1,-1)_{F_P}$. The
results that we present below in the case of $(-1,-1)_{F_P}$ are obtained when 
$F_P = \R$ using algebraic arguments, and are therefore valid in the case
of a general real closed field instead of $\R$.

Let $R$ be a real closed field, $C:=R(\sqrt{-1})$ and $H:=(-1,-1)_R$, and let
$\bbar$ denote $\id_R$, complex conjugation or quaternion conjugation,
respectively.

The theory of eigenvalues of  real (symmetric) and complex (hermitian) matrices
is well-known. Quaternionic matrices have also been studied extensively, see for
example \cite{Zhang-1997}, from which we recall the following results. 

Let $i$ and $j$ denote the generators of the quaternion division algebra $H$ over $R$.
Let $M=M_1+M_2 j\in M_n(H)$ with $M_1, M_2\in M_n(C)$. 

A quaternion $\lambda \in H$ is a \emph{left} (resp. \emph{right})
\emph{eigenvalue} of $M$ if and only if $Mx=\lambda x$ (resp. $Mx = x \lambda$),
for some $x\in H^n\sm \{0\}$.  Note that $M$ is invertible if and only if all
its left and right eigenvalues are nonzero, cf. \cite[Thm.~4.3]{Zhang-1997}. An
interesting property of the set of right eigenvalues of a quaternion matrix is
that it is closed under conjugation: If $Mx =x \lambda $ as above and $c \in H
\setminus \{0\}$, then $M(xc) = (xc)(c^{-1}\lambda c)$.

We denote by $\chi_M$ the matrix 
\[\begin{pmatrix}
M_1 & M_2\\[3pt]
-\ovl{M_2} & \ovl{M_1}
\end{pmatrix} \in M_{2n}(C)  \]
and consider the  ``characteristic polynomial'' $p_M(X):= \det (X\cdot I_{2n}
-\chi_M) \in C[X]$. The proof of \cite[Thm.~8.1(5)]{Zhang-1997} shows that
$p_M(X)$ belongs to $R[X]$. In fact, the map that
sends $M$ to $\chi_M$ is a morphism from $M_n(H)$ to its splitting $M_n(H) \ox_R C \cong
M_{2n}(C)$. Therefore, $p_M$ is the reduced characteristic polynomial of $M$, and so
has coefficients in $R$.

By \cite[Thm.~8.1(5)]{Zhang-1997} the Cayley-Hamilton theorem holds for $M$,
i.e., $p_M(M)=0$, and, for a quaternion $\lambda \in H$, $p_M(\lambda)=0$ if and
only if $\lambda$ is a right eigenvalue of $M$. 

\begin{remark}\label{ad-hoc}
  Let $M, N \in M_n(H)$.  Since the map $M \mapsto \chi_M$ is a morphism, we have
  $\chi_{MN} = \chi_M \chi_N$ and thus $p_{MN}$ is the  characteristic
  polynomial of $\chi_M \chi_N$, which is equal to the  characteristic
  polynomial of $\chi_N\chi_M = \chi_{NM}$. Therefore $p_{MN} = p_{NM}$, and in
  particular the right eigenvalues of $MN$ are exactly the right eigenvalues of
  $NM$.
  
  This can also be obtained by a direct computation, as well as the fact that,
  in case $M$ is invertible, the right eigenvalues of $M^{-1}$ are exactly the
  inverses of the right eigenvalues of $M$.
\end{remark}

The Principal Axis Theorem holds for quaternion matrices by
\cite[Cor.~6.2]{Zhang-1997}: if $M\in M_n(H)$ is such that $\ovl{M}^t = M$, then
there exist a unitary matrix $U \in M_n(H)$ (i.e., $\ovl{U}^t=U^{-1}$) and
scalars $\lambda_1,\ldots, \lambda_n \in R$ such that $\ovl{U}^t M U = \diag
(\lambda_1,\ldots, \lambda_n)$. 

A matrix $M \in M_n(H)$ such that $\ovl{M}^t=M$ is called \emph{positive
semidefinite} if $\ovl{x}^t M x \geq 0$ for all $x\in H^n$, cf.
\cite[Rem.~6.1]{Zhang-1997}. Negative semidefinite matrices are defined
similarly.

\begin{lemma}\label{ev-char}
  Let $M \in M_n(H)$. Then
  \begin{enumerate}[$(1)$] 
  
    \item If $\ovl{M}^t=M$ and $\lambda \in H$ is a right eigenvalue of $M$,
      then $\lambda \in R$.
    
    \item If $\ovl{M}^t=M$, then $M$ is positive semidefinite if and only if $M$
      only has nonnegative right eigenvalues.

    \item If $U\in M_n(H)$ is invertible, then the right eigenvalues of $M$ are
      precisely the right eigenvalues of $U^{-1} M U$.

    \item If $\ovl{M}^t=M$ and $\lambda \in H$, then $\lambda$ is a right
      eigenvalue of $M$ if and only if for every unitary $U\in M_n(H)$ such that
      $\ovl{U}^t M U$ is diagonal, $\lambda$ is one of the diagonal elements.
    
    \item If $\ovl{M}^t=M$, then $M$  is positive semidefinite if and only if
      whenever $U$ is unitary such that $\ovl{U}^t M U$ is diagonal, its
      diagonal elements are nonnegative.
 
  \end{enumerate}
\end{lemma}

\begin{proof}
  $(1)$ Let $x \in H^n \setminus \{0\}$ be such that $Mx = x\lambda$. Then
  $\lambda = (\bar{x}^tx)^{-1} (\bar{x}^t Mx) \in R$.

  $(2)$ This is \cite[Rem.~6.1]{Zhang-1997}.

  $(3)$ This is immediate.
  
  $(4)$ Let $\lambda \in H$ be a right eigenvalue of $M$.  Observe that since
  $\ovl{M}^t=M$, the standard argument shows that $\lambda \in R$. Let
  $\ovl{U}^t M U = \diag (\lambda_1,\ldots, \lambda_n)$ for $\lambda_1,\ldots,
  \lambda_n \in R$.  By $(3)$ there exists $x\in H^n\sm \{0\}$ such that
  $\ovl{U}^t M U x = x \lambda$. If, for instance, $x_1\not = 0$, we obtain
  $\lambda_1 x_1 = x_1 \lambda= \lambda x_1$. The result follows.  Conversely,
  the diagonal elements of $\ovl{U}^t M U$ are clearly right eigenvalues of
  $\ovl{U}^t M U$, and the result follows from $(3)$.

  $(5)$ This follows from $(2)$ and $(4)$.  
\end{proof}

\section*{Acknowledgments}

We would like to take this opportunity to extend a very warm thank
you to Adrian Wadsworth, who read a previous version of this paper in
great detail and discovered some serious errors. He also suggested
some more elegant arguments and greatly helped our work in
Section~\ref{sec:B-K} via a stimulating exchange of emails. In
particular, he drew our attention to the condition
$\Gamma_w=\Gamma_v$ that led to Theorem~\ref{wadth}.


\def\cprime{$'$}

\end{document}